\documentclass[11pt,letterpaper]{article}

\usepackage{amsmath,amssymb,amsfonts,amsthm,mathrsfs}
\usepackage{graphicx}
\usepackage[usenames,dvipsnames]{color}

\usepackage[margin=1in]{geometry}

\theoremstyle{plain}
\newtheorem{thm}{Theorem}
\newtheorem{lem}[thm]{Lemma}
\newtheorem{cor}[thm]{Corollary}
\newtheorem{prop}[thm]{Proposition}

\theoremstyle{definition}
\theoremstyle{remark}
\newtheorem{claim}{Claim}

\newcommand{\nat}{\ensuremath {\mathbb N} }

\newcommand{\remove}[1] {}

\newcommand{\ex} {{\bf E}}
\newcommand{\pr} {{\bf Pr}}
\newcommand{\var} {{\bf Var}}

\newcommand{\cP} {\ensuremath{\mathcal P}}

\newcommand{\sG} {\ensuremath{\mathscr G}}

\newcommand{\set}[2][]{#1\{#2 #1\}}
\newcommand{\size}[2][]{#1|#2 #1|}
\newcommand{\abs}[2][]{#1|#2 #1|}
\newcommand{\floor}[2][]{#1\lfloor#2 #1\rfloor}
\newcommand{\ceil}[2][]{#1\lceil#2 #1\rceil}
\newcommand{\paren}[2][]{#1(#2 #1)}

\newcommand{\aas}{a.a.s.}

\newcommand{\across}{m}
\newcommand{\eqdef}{:=}

\def\eps{{\epsilon}}
\def\K{{\cal K}}
\def\bin{{\bf Bin}}
\def\w{{\omega}}
\def\d{{\overline d}}

\title{Arboricity and spanning-tree packing in random graphs\\ with an application to load balancing}
\author{Pu Gao\thanks{This author has been supported by NSERC PDF.}\\\\
University of Toronto\\
{\tt pu.gao@utoronto.ca}
  \and Xavier P\' erez-Gim\' enez\thanks{Partially supported by FORMALISM TIN2007-66523.} \\\\
University of Waterloo\\
{\tt xperez@uwaterloo.ca}\\\\
  \and Cristiane M.~Sato\\\\
University of Waterloo\\
{\tt cmsato@uwaterloo.ca}
}

\remove{Pu Gao\footnote{University of Toronto. Email: {\tt pu.gao@utoronto.ca} --- This author has been
    supported by NSERC PDF.}
  \and Xavier P\' erez-Gim\' enez\footnote{University of Waterloo. Email: {\tt xperez@uwaterloo.ca} --- Partially supported by FORMALISM TIN2007-66523}
  \and Cristiane M.~Sato\footnote{University of Waterloo. Email: {\tt cmsato@uwaterloo.ca}}
   }

\date{}

\begin{document}

\maketitle

\begin{abstract}
We study the arboricity $A$ and the maximum number $T$ of edge-disjoint spanning trees of the Erd\H os-R\' enyi random graph $\sG(n,p)$.
For all $p(n)\in [0,1]$, we show  that, with high probability, $T$ is precisely the minimum between $\delta$ and $\floor{m/(n-1)}$, where $\delta$ is the smallest degree of the graph and $m$ denotes the number of edges.
Moreover, we explicitly determine a sharp threshold value for $p$ such that:
above this threshold, $T$ equals $\floor{m/(n-1)}$ and $A$ equals $\ceil{m/(n-1)}$;
and below this threshold, $T$ equals $\delta$, and we give a two-value concentration result for the arboricity $A$ in that range.
Finally, we include a stronger version of these results in the context of the random graph process where the edges are sequentially added one by one. A direct application of our result gives a sharp threshold for the maximum load being at most $k$ in the two-choice load balancing problem, where $k\to\infty$.
\end{abstract}

%
%

\remove{

TO DO LIST:

- improve (c) and (c') in the deterministic results to use later for random geometric graphs

- add a chart for Cor\ref{cor:arboricity-hitting}.

- unify the format of references. Names should be I. Lastname. Some references should be updated.

}

\section{Introduction}

\paragraph{STP number and arboricity:}

The spanning-tree packing (STP) number of a graph is the  maximum
number of edge-disjoint spanning trees it contains. Computing this parameter is a
very classical problem in combinatorial optimization.
One of the
earliest results on the STP number is a min-max relation proved by
Tutte~\cite{Tutte61} and Nash-Williams~\cite{NashWilliams61}: the STP
number of a graph is the minimum value, ranging over all partitions
$\cP$ of the vertex set, of the ratio (rounded down) between the
number of edges across $\cP$ (i.e.\ edges with ends lying in different classes of $\cP$) and $\size{\cP}-1$.
This characterisation has important consequences in computer science, where the STP number has been used as a measure of network vulnerability in case of attack or edge failure (see~\cite{Gusfield83, Cunningham85}). Intuitively speaking, it provides information about the number of edges that must be destroyed in a connected network in order
to create a given number of new components.
In addition, finding edge-disjoint spanning trees in a graph is relevant to the design of efficient and robust communication protocols (see e.g.\ a seminal article by Itai and Rodeh~\cite{Itai88}).
There are two obvious upper bounds on the STP number of a graph with $n$ vertices: the minimum degree, since each spanning tree would need to use at least one edge
incident to each vertex; and the number of edges divided by $n-1$, since each
spanning tree has exactly $n-1$ edges.
For further information on the STP number, we refer the reader to a survey by Palmer~\cite{Palmer01} on this topic.

Another closely related graph parameter that has been widely studied is the arboricity of a
graph --- i.e.~the minimum number of subforests needed to cover all of its edges.
A trivial lower bound on the arboricity  of a graph with $n$ vertices is the number of edges divided by $n-1$, since we cannot do better than covering all the edges with a set of edge-disjoint spanning trees.
Nash-Williams~\cite{Nash-Williams} also provided a min-max relation for the
arboricity of a graph,
which yields a natural interpretation of arboricity as a measure of density of the subgraphs of a graph.
This makes arboricity a useful notion in computer science, since the problem of determining the existence of dense subgraphs in large graphs is relevant to many applications in real world domains like social networking or internet computing.
In fact, finding such dense subgraphs and other related problems can often
be efficiently solved in linear time for any class of graphs with bounded arboricity (see~\cite{Chiba85,Goel06}). This includes important families of graphs such as all minor-closed classes (e.g.~planar graphs and graphs with bounded treewidth) and random graphs generated by the preferential attachment model.
Another relevant feature of arboricity is its intimate connection with the $k$-orientability of a graph and its application to certain load-balancing problems (see the discussion below on this topic).

Finding the STP number and the arboricity of a given graph are among the most
successful applications of matroids in combinatorial
optimization. Both problems can be formulated as matroid union problems
and thus can be solved in polynomial time. For more details,
see~\cite[Chapter 51]{Schrijver03}.

\paragraph{Arboricity, $k$-orientability and load balancing:}
The $k$-orientability problem consists in determining whether a graph admits an orientation of its edges so that each vertex has indegree at most $k$. As established by Hakimi~\cite{Hakimi}, the $k$-orientability of a graph is determined by the density of the densest subgraph, which brings a strong connection between $k$-orientability and arboricity.
 On the other hand, the $k$-orientability problem of a (random) graph on $n$ vertices and $m$ edges is equivalent to determining whether the maximum load is at most $k$ in the following load-balancing scenario:
$m$ balls (jobs) are assigned to $n$ bins (machines) in a way that each ball must pick between two randomly chosen bins, and we wish to minimise the load of the bins by allowing at most $k$ balls in each bin. This subject has received a lot of attention,
since the seminal result by Azar, Broder, Karlin and Upfal~\cite{Azar99} on the multiple-choice paradigm of load balancing. They consider a related online setting in which $n$ balls are sequentially thrown into $n$ bins, and each ball is allowed to choose between $h$ given random bins; and prove that the maximum load among all bins can be significantly reduced if each ball has $h\ge 2$ choices rather than one. Research on the power of choice in load balancing has since been very fruitful. We refer readers to the survey by Mitzenmacher, Richa and Sitaraman~\cite{MRS} for a detailed account of the topic. Cain, Sanders and Wormald~\cite{CSW} and Fernholz and Ramachandran~\cite{FR} simultaneously determined a threshold for a random graph (corresponding to $h=2$) to be $k$-orientable. The random hypergraph case ($h\ge 2$) was studied by a few authors due to its applications in Cuckoo hashing~\cite{DGMMPR} and disk scheduling~\cite{CSW}. See~\cite{DGMMPR,FM,FP} for $h\ge 2$ and $k=1$, and~\cite{FMP} for $h\ge 2$ and general $k$. A more general case where each ball can take $1\le w\le h$ copies and be assigned to $w$ distinct bins was studied by the first author and Wormald~\cite{GW4} for sufficiently large but constant $k$. Lelarge closed the gap for small values of $k$ in~\cite{Lelarge}. To our knowledge, all previous work focuses on constant values of $k$, and leaves open the case of $k\to\infty$. Intuitively, load should balance ``better" for larger $k$, but the previous proofs dealing with constant $k$ do not generalise to $k\to\infty$.

\paragraph{STP number and arboricity of a random graph:}
It is then relevant to study the behaviour of the STP number and the
arboricity for the Erd\H os-R\' enyi random graph $\sG(n,p)$, in which the vertex set is
$[n]$ and each of the $\binom{n}{2}$ possible edges is included independently
with probability $p$ (where $p=p(n)$ is a function of $n$).
It is a well-known fact that, for $p=(\log n -\omega(1))/n$, the random graph $\sG(n,p)$ is a.a.s.\footnote{We say that a sequence of events $H_n$ holds \emph{asymptotically almost surely} (a.a.s.) if $\lim_{n\to\infty}\pr(H_n)=1$.}\ disconnected (see e.g.\ Theorem~7.3 in~\cite{Bollobas01}), and hence the STP number is zero.
Palmer and Spencer~\cite{PalmerSpencer95} showed that a.a.s.\ the STP number of $\sG(n,p)$ equals the minimum degree whenever this has constant value $k$, which happens when $p$ is around $(\log n+(k-1)\log\log n+O(1))/n$. In fact, they proved a stronger hitting-time result in the context of the evolution of $\sG(n,p)$ when $p$ grows gradually from $0$ to $1$, and showed that a.a.s.\ the precise time when the minimum degree first becomes $k$ coincides with the time when $k$ edge-disjoint spanning trees first appear.
Moreover, Catlin, Chen and Palmer~\cite{CatlinChenPalmer93} studied the denser case of $p = C(\log n/n)^{1/3}$, where $C>0$ is a sufficiently large constant, and determined the STP number and the arboricity of $\sG(n,p)$ to be a.a.s.\ equal to $\floor{m/(n-1)}$ and $\ceil{m/(n-1)}$, respectively, where $m$ denotes the number of edges.
In a recent unpublished manuscript, Chen, Li and Lian~\cite{ChenLiLian13}
proved that, for any $(\log n +\omega(1))/n \leq p
\leq 1.1\log n/n$, a.a.s.\ the STP number of $\sG(n,p)$ equals the minimum degree. They also observed that this property a.a.s.\ does not hold  for $p\geq
51\log n/n$, and posed the question of what is the smallest value of $p$ such that the STP number of $\sG(n,p)$ differs from the minimum degree.

\paragraph{Outline of our contribution:}
%
In this paper we strengthen the previous work, and characterise the STP number and the arboricity of $\sG(n,p)$.
A direct application of our results gives a sharp threshold for the maximum load being at most $k$ in the two-choice load balancing problem, where $k\to\infty$.
Our methods rely on a successful combination of some combinatorial optimisation tools together with several probabilistic techniques, and may be hopefully extended in future research to address other relevant related problems in the area.
We first prove that for all $p(n)\in [0,1]$, the STP number is a.a.s.\ the minimum
between $\delta$ and $\floor{m/(n-1)}$, where $\delta$ and $m$ respectively denote the minimum degree and the number of edges of $\sG(n,p)$ (see \textbf{Theorem~\ref{thm:main}}). Note that the quantities $\delta$ and $\floor{m/(n-1)}$ above correspond to the two trivial upper bounds observed earlier for arbitrary graphs, so this implies that we can a.a.s.\ find a best-possible number of edge-disjoint spanning trees in $\sG(n,p)$.
Our argument uses several properties of $\sG(n,p)$ in order to bound the number of crossing edges between subsets of vertices with certain restrictions, and then applies the characterisation of the STP number by Tutte and Nash-Williams restated as Theorem~\ref{thm:nw}.
Moreover, we determine the ranges of $p$ for which the STP number takes each of these two values: $\delta$ and $\floor{m/(n-1)}$. In spite of the fact that the property $\set{\delta\le\floor{m/(n-1)}}$ is not necessarily monotonic with respect to $p$, we show that it has a sharp threshold at $p\sim\beta \log n/n$, where $\beta\approx 6.51778$ is a constant defined in \textbf{Theorem~\ref{thm:cases}}. Below this threshold, the STP number of $\sG(n,p)$ is a.a.s.\ equal to $\delta$; and above the threshold it is a.a.s.\ $\floor{m/(n-1)}$.
In particular, this settles the question raised by  Chen, Li and Lian~\cite{ChenLiLian13}.
We also include a stronger version of these results in the context of the random graph process in which $p$ gradually grows from $0$ to $1$ (or, similarly, the edges are added one by one). This provides a full characterisation of the STP number that holds a.a.s.\ simultaneously during the whole random graph process (see \textbf{Theorem~\ref{thm:process}}).
The argument combines a more accurate version of the same ideas used in the analysis of the STP number of $\sG(n,p)$ together with multiple couplings of $\sG(n,p)$ at different values of $p$.
In addition, the article contains several results about the arboricity of $\sG(n,p)$.
As an almost direct application of our result on the STP number, for $p$ above the threshold $\beta \log n/n$, we determine the arboricity of $\sG(n,p)$  to be a.a.s.\ equal to $\ceil{m/(n-1)}$. This significantly extends the range of $p$ in the result by Catlin, Chen and Palmer~\cite{CatlinChenPalmer93}. We further prove that for all other values of $p$, the arboricity of $\sG(n,p)$ is concentrated on at most two values (see \textbf{Theorem~\ref{thm:arboricity}}).
In order to prove this for the case $pn\to\infty$, we add $o(n)$ edges to $\sG(n,p)$ in a convenient way that guarantees a full decomposition of the resulting graph into edge-disjoint spanning trees. This construction builds upon some of the ideas that we previously use to study the STP number.
The case $pn=O(1)$ uses different proof techniques which rely on the structure of the $k$-core of $\sG(n,p)$ together with the Nash-Williams characterisation of arboricity restated as Theorem~\ref{thm:nw-arb}.
Finally, some of the aforementioned results on the arboricity are also given in the more precise context of the random graph process (see \textbf{Theorem~\ref{thm:arboricity-hitting}} and \textbf{Corollary~\ref{cor:arboricity-hitting}}), similarly as we did for the STP number.
As a direct corollary of our result on arboricity, we determine a sharp threshold for the $k$-orientability of $\sG(n,m)$ where $k\to\infty$ (see \textbf{Theorem~\ref{thm:loadbalancing}}). This successfully settles the load-balancing problem in the scenario where $m=\omega(n)$ balls are allocated into $n$ bins, and each ball has two choices, uniformly at random chosen from $[n]$. We prove that in this case, balls can be allocated so that most bins receive an almost even load. This extends the result by Cain, Sanders and Wormald~\cite{CSW} and Fernholz and Ramachandran~\cite{FR} to the case of $k\to\infty$.

\paragraph{Related work:}
The behaviour of the STP number and the arboricity has been also studied in other models of
random graphs. Frieze and \L uczak~\cite{FriezeLuczak90} considered the
random directed graph in which each vertex chooses $k$ out-neighbours
uniformly at random, with fixed~$k$. This graph has $k$ disjoint
spanning trees with probability going to $1$ (where the orientation of
the arcs is ignored). Some variants of arboricity have also been
studied. The linear arboricity of a graph is the minimum number of
forests consisting only of paths needed to cover all edges of the
graph. This parameter was studied by McDiarmid and
Reed~\cite{McDiarmidReed90} for random regular graphs.

%
%

\section{Main results}\label{sec:results}

Let $\sG(n,p)$ denote the random graph with vertex set $[n]$ such that
each possible edge in $\set{\set{u,v}:u,v\in [n],\ u\neq v}$ is
included independently with probability~$p$.
In this article, we regard $p$ as a function of $n$, and consider asymptotic statements as $n\to\infty$.
 Given a sequence of
events $(H_n)_{n\in \nat}$, we say that $H_n$ happens \emph{asymptotically
almost surely} (\aas)  if $\pr(H_n)\to 1$ as $n\to\infty$.
Given real sequences $a_n$ and $b_n$ (possibly taking negative
values), we write: $a_n=O(b_n)$ if there is a constant $C>0$ such that
$|a_n|\le C |b_n|$ for all $n$;  $a_n=o(b_n)$ if eventually $b_n\ne0$ and
$\lim_{n\to\infty} a_n/b_n=0$; $a_n=\Omega(b_n)$ if eventually $a_n\ge
0$ and $b_n=O(a_n)$; $a_n=\omega(b_n)$ if eventually $a_n\ge 0$ and $b_n=o(a_n)$;
$a_n=\Theta(b_n)$ if eventually $a_n\ge 0$, $a_n=O(b_n)$ and $a_n=\Omega(b_n)$; and finally $a_n\sim b_n$ if $a_n=(1+o(1))b_n$.
In particular, in this paper, all constants involved in these notations do not depend on the $p$ under discussion. For instance, if we have $a_n=\Omega(b_n)$, where $b_n$ may be an expression involving $p=p(n)$, then it means that there are constants $C>0$ and $n_0$ (both independent of $p$), such that $a_n\ge C|b_n|$ uniformly for all $n\ge n_0$ and for all $p$ in the range under discussion.
We also assume in this paper that $n$ is greater than some absolute constant (e.g.\ in a few places we require $n\ge 2$), which does not depend on any variables under discussion. We often omit this assumption in the statements of lemmas and theorems.

Given a graph $G$, let $m(G)$ be the number of edges
of $G$ and let $\delta(G)$ denote the minimum degree of $G$.
Let $T(G)$ be the STP number of $G$ --- i.e.~the maximum number of
edge-disjoint spanning trees in~$G$ (possibly $0$ if $G$ is
disconnected).
\begin{thm}
  \label{thm:main}
  For every $p = p(n)\in [0,1]$, we have that \aas
  \begin{equation*}
    T(\sG(n,p)) =
    \min\set[\bigg]{
      \delta(\sG(n,p)),
      \floor[\bigg]{\frac{m(\sG(n,p))}{n-1}}}.
  \end{equation*}
\end{thm}

\begin{thm}
  \label{thm:cases}
  Let $\beta=2/\log(e/2)\approx 6.51778$.  Then
  \begin{itemize}
  \item[(i)] if $p = \frac{\beta(
    \log n-\log\log n/2) -\omega(1)}{n-1}$, then \aas\
    $\delta(\sG(n,p)) \le \floor[\big]{\frac{m(\sG(n,p))}{n-1}}$ and so
    $T(\sG(n,p)) = \delta(\sG(n,p))$;
  \item[(ii)] if $p = \frac{\beta
    (\log n -\log\log n/2) +\omega(1)}{n-1}$, then \aas\ $\delta(\sG(n,p)) >
    \floor[\big]{\frac{m(\sG(n,p))}{n-1}}$ and so $T(\sG(n,p)) =
    \floor[\big]{\frac{m(\sG(n,p))}{n-1}}$.
  \end{itemize}
\end{thm}

Consider the random graph process $G_0,G_1,\ldots,G_{\binom{n}{2}}$ defined as follows: for each $m=0,\ldots,\binom{n}{2}$, $G_m$ is a graph with vertex set $[n]$; the graph $G_0$ has no edges; and, for
each $1\le m\le \binom{n}{2}$, the graph $G_{m}$ is obtained by
adding one new edge to $G_{m-1}$ chosen uniformly at random among the
edges not present in $G_{m-1}$. Equivalently, we can choose uniformly at random a permutation $(e_1,\ldots,e_{\binom{n}{2}})$ of the edges of the complete graph with vertex set $[n]$, and define each $G_m$ to be the graph on vertex set $[n]$ and edges $e_1,\ldots,e_m$.

The following theorem is a strengthening of Theorems~\ref{thm:main} and~\ref{thm:cases} in the context of the random graph process just described. Note that the a.a.s.\ statements in Theorem~\ref{thm:process} refer to events that hold \textbf{simultaneously} for all $m$ in a certain specified range, as we add edges one by one.
\begin{thm}
  \label{thm:process}

Let $\beta=2/\log(e/2)\approx 6.51778$. The following holds in the random graph process $G_0,G_1,\ldots,G_{\binom{n}{2}}$.
\begin{itemize}
\item[(i)]A.a.s.\ $T(G_m)=\min\set[\big]{\delta(G_m),
    \floor[\big]{m/(n-1)}}$ for every
  $0\le m\le\binom{n}{2}$.
\item[(ii)]
Moreover, for any constant $\epsilon>0$, a.a.s.\
\begin{itemize}
\item[$\bullet$] $\delta(G_m) \le \floor[\big]{m/(n-1)}$ for every
  $0\le m\le  \frac{(1-\epsilon)\beta}{2} n \log n$, and
\item[$\bullet$] $\delta(G_m) > \floor[\big]{m/(n-1)}$ for every
  $\frac{(1+\epsilon)\beta}{2} n \log n\le m\le\binom{n}{2}$.
\end{itemize}
\end{itemize}
\end{thm}
\paragraph{Remark:} We will actually prove a stronger result than Theorem~\ref{thm:process} (ii). See Theorem~\ref{thm:continuousCases} in Section~\ref{sec:process}.

For any graph $G$, let $A(G)$ denote the arboricity of $G$, i.e.~the minimum number of
subforests of $G$ which cover the whole edge set of $G$.
\begin{thm}
  \label{thm:arboricity}
  Let $\beta=2/\log(e/2)\approx 6.51778$.
  \begin{enumerate}
  \item[(i)] For all $p=\frac{\beta (\log n-\log\log n/2)+\omega(1)}{n-1}$, a.a.s.\
    $A(\sG(n,p))=\ceil[\big]{\frac{m(\sG(n,p))}{n-1}}$;
  for all $p=\omega(1/n)$, a.a.s.\ $A(\sG(n,p))\in
    \{\ceil[\big]{\frac{m(\sG(n,p))}{n-1}},\ceil[\big]{\frac{m(\sG(n,p))}{n-1}}+1\}$;
 \item[(ii)] For all $p=\Theta(1/n)$, a.a.s.\
$A(\sG(n,p))= (1+\Theta(1))pn/2$. Moreover, there exists a $k>0$ (depending on $p$), such that a.a.s.\ $A(\sG(n,p))\in \{k,k+1\}$.
  \item[(iii)] If $p=o(1/n)$, then a.a.s.\ $A(\sG(n,p))\le1$.
  \end{enumerate}
\end{thm}
\paragraph{Remark:} (a) Note that Theorem~\ref{thm:arboricity} (i) is a simple corollary of Theorem~\ref{thm:arboricity-hitting} below. Indeed, for most values of $p=\omega(1/n)$, we can do even better than the two-value concentration result stated above and a.a.s.\ determine the exact value of the arboricity (cf. the remark that follows Theorem~\ref{thm:arboricity-hitting}).

\noindent (b) It follows from Theorem~\ref{thm:arboricity} that, for all $p=\omega(1/n)$, the arboricity of $\sG(n,p)$ is asymptotic to $pn/2$, whereas this property fails for $p=O(1/n)$.

\begin{thm}\label{thm:arboricity-hitting}
Let $\beta=2/\log(e/2)\approx 6.51778$. The following holds in the random graph process $G_0,G_1,\ldots,G_{\binom{n}{2}}$. \begin{enumerate}
  \item[(i)] Let $m_0$ be any function of $n$ such that
    $m_0/n\to\infty$ and let $\epsilon>0$ be any constant. Then,
    a.a.s.\ simultaneously for all $m\ge m_0$ such that $\delta(G_m)
    \le m/(n-1)$,
\begin{equation}\label{eq:AGPhi}
  \ceil[\bigg]{\frac{m+\phi_1}{n-1} } \le A(G_m) \le \ceil[\bigg]{\frac{m+\phi_2}{n-1}},
\end{equation}
where $\phi_1=n/\exp\left(\frac{(1+\eps)}{\beta}\frac{2m}{n}\right)=o(n)$ and $\phi_2=n/\exp\left(\frac{(1-\eps)}{\beta}\frac{2m}{n}\right)=o(n)$. In particular, a.a.s.\ $A(G_m) \in \{ \ceil{\frac{m}{n-1}}, \ceil{\frac{m}{n-1}}+1\}$ for all $m$ in that range.
\item[(ii)] Moreover, a.a.s.\ simultaneously for every $m$ such that $\delta(G_m)\ge m/(n-1)$ we have
\[
A(G_m)= \ceil[\big]{m/(n-1)}.
\]
 \end{enumerate}
\end{thm}
\paragraph{Remark:} Given positive integers $a$ and $b$, let $R(a,b)=a-b\floor{a/b}$ denote the remainder of $a$ divided by $b$. Then~\eqref{eq:AGPhi} in Theorem~\ref{thm:arboricity-hitting}(i) implies $A(G_m) = \ceil{\frac{m}{n-1}}$ for those $m$ such that $0<R(m,n-1)\le n-1-\phi_2$ (which is the case for most values of $m$), and $A(G_m) = \ceil{\frac{m}{n-1}}+1$ if $R(m,n-1)=0$ or $R(m,n-1)> n-1-\phi_1$. For those few remaining values of $m$ such that $n-1-\phi_2< R(m,n-1)\le n-1-\phi_1$ we can only say $A(G_m) \in \{ \ceil{\frac{m}{n-1}}, \ceil{\frac{m}{n-1}}+1\}$.
\begin{cor}\label{cor:arboricity-hitting}
Let $m_{A=i}$ denote the minimum $m$ such that $A(G_m)$ becomes $i$ in the random graph process $G_0,G_1,\ldots,G_{\binom{n}{2}}$. Let $i_0$ be any function of $n$ such that $i_0\to\infty$ and $\epsilon>0$ be a constant. Then a.a.s.\
 \begin{enumerate}
  \item[(i)] for every $i_0\le i\le (1-\epsilon)\beta\log n/2$,
\[
(i-1)(n-1) -  \phi_2 <   m_{A=i} < (i-1)(n-1) - \phi_1,
\]
where $\phi_1=n/\exp\left(\frac{2(1+\eps)}{\beta}i\right)=o(n)$ and $\phi_2=n/\exp\left(\frac{2(1-\eps)}{\beta}i\right)=o(n)$; and
   \item[(ii)] for every $(1+\epsilon)\beta\log n/2\le i\le n/2$,
\[
m_{A=i}=(i-1)(n-1)+1.
\]
 \end{enumerate}
\end{cor}
\paragraph{Remark:} Indeed, as shown in the proof in Section~\ref{sec:arboricity}, $m_{A=i}>(i-1)(n-1) -  \phi_2 $ holds for all $i_0\le i\le (1+\epsilon)\beta\log n/2$.

Recall that a graph $G$ is $k$-orientable if all edges of $G$ can be oriented so that the maximum indegree of the oriented graph is at most $k$. It was shown by Hakimi~\cite{Hakimi} that $G$ is $k$-orientable if and only if it contains no subgraph with average degree more than $2k$. Trivially, this implies that no graph on $n$ vertices with more than $kn$ edges can be $k$-oriented. Moreover, by the previous corollary together with a result by Nash-Williams (see Theorem~\ref{thm:nw-arb}), we obtain the following theorem, which characterises the $k$-orientability of $\sG(n,m)$ (i.e.~the uniform random graph on $n$ vertices and $m$ edges) for $k\to\infty$.
\begin{thm}\label{thm:loadbalancing}
Let $f$ be any function of $n$ that goes to infinity as $n\to\infty$. Then for every integer $k\ge f$ and any $\eps>0$, as $n\to\infty$,
\[
\pr(\sG(n,m)\ \mbox{is $k$-orientable})\to
\begin{cases}
1& \text{if $m\le k(n-1)-\phi$}\\
0& \text{if $m\ge kn+1$},
\end{cases}
\]
where $\phi=0$ if $k\ge \frac{1+\eps}{2}\beta\log n$ and $\phi=n/\exp\left(\frac{2(1-\eps)}{\beta}k\right)=o(n)$ if $f\le k< \frac{1+\eps}{2}\beta\log n$.
\end{thm}
\paragraph{Remark:}
In particular, Theorem~\ref{thm:loadbalancing} implies that the property of being $k$-orientable has a sharp threshold in $\sG(n,m)$ at $m\sim kn$, since we are assuming $k\to\infty$. By a closer inspection of Corollary~\ref{cor:arboricity-hitting}, we may obtain more accurate bounds on the critical $m$ for $k$-orientability by distinguishing each of the two cases $f\le k\le \frac{1-\eps}{2}\beta\log n$ and $k\ge \frac{1+\eps}{2}\beta\log n$.
Moreover, we recall that the $k$-orientability of $\sG(n,m)$ can be interpreted in terms of a load-balancing problem in a setting where $m=\omega(n)$ jobs (edges) are assigned to $n$ machines (vertices), in a way that each job is allocated to one machine, selected from two randomly given ones.
 In this context, we want to minimise the maximum load (number of jobs) received by any machine. Observe that Theorem~\ref{thm:loadbalancing} implies that there exists a load distribution such that the maximum load is around $m/n$, and almost all machines receive a load very close to the maximum load.
\medskip

The paper is organised as follows. We first introduce some basic tools in Section~\ref{sec:tools}, including two classic theorems by Tutte and Nash-Williams that characterise the spanning tree-packing number and the arboricity of a graph. There, we also prove two deterministic results (Propositions~\ref{prop:a} and~\ref{prop:b}) that are central in our argument. They give a list of conditions under which the STP number equals the minimum between $\delta$ and $\floor{m/(n-1)}$.
In Section~\ref{sec:propGnp}, we prove several lemmas about properties of $\sG(n,p)$. These lemmas will be used in conjunction with the two aforementioned deterministic propositions to derive most of the main results in the paper.
Finally, we prove Theorem~\ref{thm:main} in Section~\ref{sec:main}, Theorem~\ref{thm:cases} in Section~\ref{sec:cases}, Theorem~\ref{thm:process} in Section~\ref{sec:process}, Theorem~\ref{thm:arboricity-hitting} in Section~\ref{sec:arboricity} and Theorem~\ref{thm:arboricity} in Section~\ref{sec:sparse}.


\section{Deterministic tools}
\label{sec:tools}
In this section, we introduce some basic tools that lie in the core of our argument.
Given a graph $G$, let $V(G)$ denote the vertex set of~$G$ and let
$E(G)$ denote the edge set of~$G$. Recall that $m(G)$ is the number of edges, and $\delta(G)$ is the minimum degree  of $G$.
If $|V(G)|\ge2$, define
$\d(G) \eqdef 2m(G)/(|V(G)|-1)$.  Note that $\d(G)$ differs from the
average degree of $G$ by a small factor of $|V(G)|/(|V(G)|-1)$. Also, let $t(G)=\min\{\delta(G), \d(G)/2\}$.

We first restate two well-known results by Tutte and Nash-Williams that provide a useful characterisation of the STP number and the arboricity of a graph $G$.
For any partition $\cP$ of the vertex set $V(G)$ of a graph $G$, let
$\across(\cP)$ denote the number of edges in $G$ with ends in
distinct parts of $\cP$.
\begin{thm}[Tutte~\cite{Tutte61} and
Nash-Williams~\cite{NashWilliams61}]
\label{thm:nw}
Let $G$ be a graph and $t$ a positive integer. Then $G$ contains $t$ edge-disjoint spanning trees
if and only if, for every partition $\cP$ of the vertex set of $G$
such that every class is non-empty,
  \begin{equation}
    \label{eq:nw_condition}
    \across(\cP)
    \geq
    t (|\cP|-1).
  \end{equation}
\end{thm}
For any $S\subseteq V(G)$, let $E[S]$ denote the set
of edges of $G$ with both ends in~$S$.
\begin{thm}[Nash-Williams~\cite{Nash-Williams}]
\label{thm:nw-arb}
Let $G$ be a graph  and $t$ a positive integer. Then the edge set of $G$ can be covered by $t$
forests if and only if, for every nonempty subset $S$ of vertices
of~$G$,
\begin{equation}
    \label{eq:nw_condition_arb}
    \size{E[S]}
    \leq
    t (\size{S}-1).
  \end{equation}
\end{thm}

The next two propositions play a central role in this paper. They make use of Theorem~\ref{thm:nw} to determine the STP number of any well-behaved graph satisfying certain conditions.
For $\epsilon > 0$, we say that a vertex of $G$ is {\em $\eps$-light} if its
degree is at most $\delta(G)+\epsilon \d(G)$.
\begin{prop}\label{prop:a}
  Let $G=G_n$ be a graph on vertex set $[n]$. Let $\delta\eqdef
  \delta(G)$ and let $\d\eqdef \d(G)$. Suppose that $\d\to\infty$ as
  $n\to\infty$ and that there exist constants $\eps,\zeta,\eta>0$ such
  that the following hold, for all sufficiently large $n$.
\begin{itemize}
\item[(a)] The minimum degree $\delta$ is at most $(\epsilon/4)\d$; there is no
  pair of adjacent $\epsilon$-light vertices; and all vertices of
  $G$ have at most one $\epsilon$-light neighbour.
\item[(b)] No set of size $s<\zeta n$ induces more than $(\epsilon/4)\d s$ edges.
\item[(c)] For all disjoint $S,S'\subseteq [n]$ with $\size{S}\ge\size{S'}\ge
  \zeta n$, we have that $\size{E(S, S')}\geq \eta \d n$.
\end{itemize}
Then eventually $T(G)=\delta$.
\end{prop}
\begin{prop}\label{prop:b}
Let $G=G_n$ be a graph on $[n]$. Let $\delta\eqdef \delta(G)$ and $\d\eqdef \d(G)$, and suppose that $\d\to\infty$ as $n\to\infty$. Let $t=\min\{\delta,\d/2\}$. Suppose moreover that there exist constants
$0<\eps,\eta,\zeta \le 1$ such that the following hold, for sufficiently large $n$.
\begin{itemize}
\item[(a')] Either we have that $\delta > \frac{(1+\epsilon)\d}{2}$; or there are no
  adjacent $\eps$-light vertices and each vertex of $G$ is adjacent to at most one $\eps$-light vertex.
\item[(b')] For all $S\subseteq V(G)$, with $\size{S}\ge \zeta n$,
  we have that $d(S)\geq \d(1-o(1))$, where $d(S)$ denotes the sum of degrees
  of vertices in $S$ divided by $\size{S}$.
\item[(c')] For all disjoint $S,S'\subseteq V(G)$ with $\size{S} \ge\size{S'}\ge \zeta n$,
  we have that $\size{E(S, S')}\geq \eta \d\size{S}\size{S'}/n$.
\item[(d')] For all $\emptyset\subsetneq S\subsetneq V(G)$, we have that $\size{E(S,
    \overline{S})}\geq t$.
\item[(e')] No set of size $s<\zeta n$ induces more than $(\epsilon/4) ts$ edges.
\end{itemize}
Then eventually $T(G)=\floor t$.
\end{prop}
Propositions~\ref{prop:a} and~\ref{prop:b} will be used to determine the STP number and the arboricity of $\sG(n,p)$ (see the arguments leading to the proofs of Theorems~\ref{thm:main}, \ref{thm:process}, \ref{thm:arboricity} and~\ref{thm:arboricity-hitting}).
Basically, according to the range of $p$, we will show that $\sG(n,p)$ (or some modification of $\sG(n,p)$) satisfies the conditions in Proposition~\ref{prop:a} or Proposition~\ref{prop:b} with sufficiently high probability.
Proposition~\ref{prop:a} is applied when the minimum degree is relatively small compared to $\d$, whereas otherwise Proposition~\ref{prop:b} is used instead.
Thus, we need a good estimation of $\delta(\sG(n,p))$; together with several graph-expansion-related properties, as required by conditions (b), (c), (c'), (d') and (e');
and also some properties about the $\eps$-light vertices addressed in conditions (a) and (a').
In the following section, we derive bounds on the probability that these properties hold in $\sG(n,p)$ for some relevant ranges of $p$.

\begin{proof}[Proof of Proposition~\ref{prop:a}]
  We will show that every partition of the vertices of $G$
  satisfies~\eqref{eq:nw_condition} with $t = \delta$, and thus $G$
  has $\delta$ edge-disjoint spanning trees by Theorem~\ref{thm:nw}.

  Let $\cP$ be a partition of $V(G)$. Parts of size one are denoted
  \emph{singletons}, and singletons consisting of one $\epsilon$-light
  vertex are called \emph{$\epsilon$-light singletons}. We may assume that
\begin{equation}
\mbox{ every
  part with size at least $2$ has one vertex that is not
  $\epsilon$-light.} \label{hypothesis}
  \end{equation}

  This is because, given a part of size at least $2$
  and with only $\epsilon$-light vertices, we can refine the partition
  by turning each vertex in this part into a singleton, and this
  increases the number of parts without increasing the number of edges
  with ends in distinct parts by Condition~(a).

  Let $\K_1$ denote the set of $\epsilon$-light singletons, let $\K_2$
  denote the set of parts of size between $2$ and $\zeta n$ together
  with the singletons that are not $\epsilon$-light, and let $\K_3$
  denote the set of other parts. For $i=1,2,3$, let
  $k_i=|\K_i|$. Then, $|\cP|=k_1+k_2+k_3$.

  By Condition~(a), no $\epsilon$-light vertices are adjacent. Thus,
  the number of edges incident with a vertex in $\K_1$ is at least
  $\delta k_1$. Suppose $\K_2$ is non-empty and suppose $S$ is a part in $\K_2$, and let $r$ be the
  number of vertices in $S$ that are not $\epsilon$-light. By the assumption in~\eqref{hypothesis} and
  the definition of $\K_2$, we must have $1\le r\le\zeta n$. The number of edges
  between these $r$ vertices is at most $(\epsilon/4) \d r$ by
  Condition~(b). Since these vertices are not $\epsilon$-light, each
  of them has degree at least $\delta+\eps \d$. By Condition (a), each
  of these vertices is adjacent to at most one $\epsilon$-light
  vertex. Thus,
\begin{equation*}
  \size{E(S,\overline S\setminus \K_1)}
  \geq
  r(\delta+\eps \d-1)-2(\eps/4)\d r  \geq \delta + (\eps/4)\d,
\end{equation*}
where the term $-1$ in the first inequality accounts for a possible
$\epsilon$-light neighbour of each one of these $r$ vertices and we
use the fact that $d\to\infty$. Thus, the number of edges in the
partition $\cP$ is at least
\begin{equation}\label{eq:edges1}
\across(\cP)\ge \delta k_1 +\frac{\delta +(\epsilon/4)\d}{2} k_2 \ge \delta(k_1 + k_2),
\end{equation}
as $\delta\le(\eps/4)\d$ by Condition~(a). If $k_3\leq 1$, this
already shows that $\across(\cP)\geq \delta
(\size{\cP}-1)$. Otherwise, we have $2\le k_3\le 1/\zeta$, and the
number of edges between any two parts of $\K_3$ is at least $\eta \d n$
by Condition~(c). We can add these additional edges
to~\eqref{eq:edges1} and obtain
\begin{equation}\label{eq:edges2}
  \across(\cP)  \geq  \delta (k_1 + k_2)  + \eta \d n \geq \delta (k_1+k_2) + (\epsilon/4)\d k_3 \geq \delta \size\cP,
\end{equation}
since eventually $\eta n \ge (\epsilon/4)/\zeta \ge (\epsilon/4 ) k_3$, and $\delta\le(\eps/4)\d$.
\end{proof}
\begin{proof}[Proof of Proposition~\ref{prop:b}]
We will show that every partition of the vertices of $G$ satisfies~\eqref{eq:nw_condition}, and thus $G$ has $\floor t$ edge-disjoint spanning trees by Theorem~\ref{thm:nw}.

We say that a set $S\subseteq V$ is \emph{large} if $\size{S}\geq
\zeta n$. We say that a partition of $V$ is \emph{simple} if each
class either is large or a singleton (that is, it consists of a single
vertex). Recall that $\across(\cP)$ denotes the number of edges with
ends in distinct parts of $\cP$.
\begin{claim}
\label{claim:simple}
  If $\cP$ is a simple partition, then $\across(\cP)\geq t(\size{\cP}-1)$.
\end{claim}
Assume Claim~\ref{claim:simple} holds, and suppose for a contradiction
that there is a partition $\cP$ of $V$ such that $\across(\cP)
<t(\size{\cP}-1)$.  By Claim~\ref{claim:simple}, $\cP$ is not a simple
partition. Given a set $S$ and a vertex $v\in S$, let $d_S(v)$ denote
the number of neighbours of $v$ inside $S$. Since $\cP$ is not simple,
we can find a non-large part $S$ of $\cP$ with at least $2$
vertices. By Condition~(e') and since $\eps\le1$, $S$ must contain one vertex $w$ with
\begin{equation}\label{eq:dsw}
d_S(w) \le \frac{2|E[S]|}{|S|} \le \frac{2\epsilon t|S|}{4|S|} \le t/2.
\end{equation}
Moreover, condition (d') implies that
\begin{equation}\label{eq:halfratio}
\across(\cP)\geq (t/2)(\size{\cP}-1).
\end{equation}
Let $\cP'$ be obtained from $\cP$ by turning $w$ into a singleton. We
have $|\cP'|=|\cP|+1$ and
$\across(\cP')=\across(\cP)+d_S(w)$. Combining this facts together
with~\eqref{eq:dsw} and~\eqref{eq:halfratio}, we obtain
\begin{equation}
  \label{eq:ratioworse}
  \frac{\across(\cP')}{\size{\cP'}-1}
  =
  \frac{\across(\cP)+d_S(w)}{\size{\cP}}
  \leq
  \frac{\across(\cP)+t/2}{\size{\cP}}
  \leq
  \frac{\across(\cP)+\frac{\across(\cP)}{\size{\cP}-1}}{\size{\cP}}
  =
  \frac{\across(\cP)}{\size{\cP}-1}.
\end{equation}
Repeat this procedure of turning vertices into singletons until no
parts of size between $2$ and $\zeta n$ remain, and therefore obtain a
simple partition $\cP''$. Since~\eqref{eq:ratioworse} holds in each
iteration, we have
\begin{equation*}
  \frac{\across(\cP'')}{\size{\cP''}-1}
  \leq
  \frac{\across(\cP)}{\size{\cP}-1}
  < t,
\end{equation*}
which contradicts Claim~\ref{claim:simple}.

To complete the argument, we proceed to prove
Claim~\ref{claim:simple}. Let $\cP$ be a simple partition. If all
parts of $\cP$ are singletons, then we have $\across(\cP) =
\frac{\d}{2}(n-1) = \frac{\d}{2}(\size{\cP}-1)$.  Suppose otherwise
there is at least one large part. Since $\cP$ is simple, each large
part has at least $\zeta n$ vertices and so there are at most
$\ell\eqdef 1/\zeta = O(1)$ large parts. Let $k$ be the number of
singletons in~$\cP$. Note that $k\leq (1-\zeta)n$ since any large part
has at least $\zeta n$ vertices.

  Suppose first that $\zeta n\le k\le(1-\zeta)n$. Then the
  average degree of the singletons is at least $\d(1-o(1))$ by
  Condition~(b'). Since there is at least one large part, the number of
  edges between the $k$ singletons and this large part is at least
  $\eta\zeta \d k$ by Condition~(c'). Hence, $\across(\cP)$ is at least the number of edges incident with a singleton, which is at least
  \[
  \frac{\d(1-o(1))k+\eta\zeta \d k}{2} \ge
  \frac{(1+\eta\zeta/2)k}{k+\ell-1}(\d/2)(k+\ell-1).
  \]
  This satisfies Equation~\eqref{eq:nw_condition} with $\d/2\ge t$ for
  large enough $n$, since $k\ge \zeta n$ and $\ell=O(1)$.

  Suppose otherwise that $0\le k\le \zeta n$. By Condition~(a'), we have
  that either $\delta > \frac{(1+\epsilon)\d}{2}$; or there are no adjacent
  $\eps$-light vertices and each vertex is adjacent to at most one $\eps$-light vertex. The number of edges between singletons is at most $(\eps/4)tk\le\epsilon tk$
  by Condition~(e'). In the first case where $\delta > \frac{(1+\epsilon)\d}{2}$, the total number of edges
  incident to the singletons is at least
\begin{equation}\label{eq:dk2}
    \frac{(1+\epsilon)\d}{2}k - \epsilon tk \ge (1+\epsilon)tk - \epsilon tk \ge tk.
  \end{equation}

 Now we consider the second case.
  Recall that a vertex is $\eps$-light if it has degree at most
  $\delta+\epsilon \d$. Suppose that there are no adjacent
  $\eps$-light vertices and each vertex is adjacent to at most one $\eps$-light vertex.  Let $K_1$ denote the set of singletons that are $\eps$-light and $K_2$
  the set of other singletons (singletons that are not $\eps$-light). Let $k_i=\size{K_i}$ for $i=1,2$, so $k=k_1+k_2$ (possibly
  $k_1,k_2=0$). Since there are no adjacent $\eps$-light vertices, $|E(K_1,\overline{K_1})|\ge \delta k_1$. Since no two $\eps$-light vertices have a common neighbour, we have
$d_{[n]\setminus K_1}(v)\ge \delta+\eps \d-1$, for every $v\in K_2$. Moreover, Condition~(e') guarantees that
  there are at most $\eps t k_2 \le \eps \d k_2/2$ edges inside $K_2$, and therefore
  $|E(K_2,\overline{K_2}\setminus K_1)|\ge (\delta+\epsilon \d-1)k_2 - \epsilon \d k_2/2$.
 Thus, the total number
  of edges incident with singletons is at least
  \begin{equation}\label{eq:deltak}
    \delta k_1 + (\delta+\epsilon \d-1)k_2 - \epsilon \d k_2/2
    \ge \delta k \ge tk,
  \end{equation}
  eventually as $\d=\omega(1)$ by our assumption.
  Thus, we have proved that in both cases, the number of edges incident with singletons is at least $tk$.
If the number of large parts is exactly $1$, then~(\ref{eq:nw_condition}) holds as $|\cP|=k+1$ and $\across(\cP)\ge tk$ by~\eqref{eq:dk2} and~\eqref{eq:deltak}. Otherwise, if there are at least two large parts, the number of edges  between any
  two of them is at least $\eta\zeta^2\d n$  by Condition~(c'). Thus, for large enough $n$,
  \begin{equation*}
    \across(\cP)
    \ge
    tk +
    \eta  \zeta^2 \d n
    \ge
    t\left(k +
    \frac{\eta\zeta^2\d n}{t}\right)
    \ge
    t (k+\ell-1),
  \end{equation*}
  since $t\le \d/2$ and $\ell=O(1)$.
\end{proof}
%
%
%
\section{Properties of $\sG(n,p)$}\label{sec:propGnp}

In this section, we always let $G$ denote $\sG(n,p)$, and
let $\delta \eqdef \delta(\sG(n,p))$, $m\eqdef m(\sG(n,p))$ and $\d\eqdef \d(\sG(n,p))=2m/(n-1)$. Recall our earlier assumption that $n\ge2$. So $\d$ is well defined.  For
any vertex $v$, let $d_v$ denote the degree of $v$ in $G$.

\subsection{Typical degrees}

Our aim here is to show that $m$ and $\d$ are a.a.s.\ concentrated around their expected values, and that most of the vertices of $\sG(n,p)$ have degree close to $\d$.
To do so, we first state a version of the well-know Chernoff's bounds (see~e.g.~Theorems~4.4 and~4.5 in~\cite{MitzenmacherUpfal})
\begin{thm}[Chernoff's bounds] \label{thm:chernoff} Let $X_1,\ldots,X_n$ denote $n$ independent Bernoulli variables. Let $X=\sum_{i=1}^n  X_i$ and let $\mu=\ex X$. Then for any $0<\tau<1$,
$$
\pr\left(X\ge (1+\tau)\mu\right)\le \exp(-\tau^2\mu/3),\quad \pr\left(X\le (1-\tau)\mu\right)\le \exp(-\tau^2\mu/2).
$$
\end{thm}

\begin{lem}
  \label{lem:avgdeg}
  For any function $\tau(n) < 1$, we have that the probability that $|\d
  - pn| \leq \tau pn$ and $|m - p\binom{n}{2}| \leq \tau p\binom{n}{2}$ is at
  least $1-2\exp(-A \tau^2pn^2)$ where $A=1/12$.
\end{lem}
\begin{proof}
By the definition of $\d$, the events $\abs{\d-pn}> \tau pn$ and $\abs{m-p\binom{n}{2}}> \tau p\binom{n}{2}$ are equivalent.
Then, since the number of edges in $\sG(n,p)$ is distributed as $\bin(\binom{n}{2},p)$, we apply Chernoff's bound in Theorem~\ref{thm:chernoff} and obtain
 \[
     \pr\paren[\Bigg]{\abs[\big]{m-p\binom{n}{2}}> \tau p\binom{n}{2}}
\leq  2\exp\paren[\Bigg]{-\frac{\tau^2p\binom{n}{2}}{3}}.
\qedhere
\]
\end{proof}

\begin{lem}
  \label{lem:deg}
  Let $f\ge0$ be any function of $n$ such that $f\to\infty$. Then,
  there exists a constant $C>0$ such that for every $f/n\le p\le 1$ the
  following holds in $\sG(n,p)$ with probability at least
  $1-e^{-C(pn)^{1/3}}$. The number of vertices with degree not in
  $[\d-(pn)^{2/3},\d+(pn)^{2/3}]$ is at most $n/e^{Cf^{1/3}}$.
\end{lem}
\begin{proof}
  We have that
  \begin{equation*}
    \pr\paren{|d_v-\d|>(pn)^{2/3}}
    \leq
    \pr\paren[\Big]{|\d-pn|>\frac{(pn)^{2/3}}{2}}
    +
    \pr\paren[\Big]{|d_v-pn|>\frac{(pn)^{2/3}}{2}}.
  \end{equation*}
  By Lemma~\ref{lem:avgdeg} with $\tau=\frac{1}{2}\cdot (pn)^{-1/3}$,
  \begin{equation*}
    \pr\paren[\Big]{ |\d
      - pn| \leq \frac{(pn)^{2/3}}{2}}
    \leq
    2\exp(-A n\cdot (pn)^{1/3}),
  \end{equation*}
  where $A$ is a positive constant. By Chernoff's inequality in Theorem~\ref{thm:chernoff}, for a
  positive constant $B$,
  \begin{equation*}
    \pr\paren[\Big]{|d_v-pn|>\frac{(pn)^{2/3}}{2}}
    \leq
    2\exp\paren[\Big]{-\frac{Bpn}{(pn)^{2/3}}}
    =
    2\exp\paren[\Big]{-B(pn)^{1/3}}.
  \end{equation*}
  Thus, there is a positive constant $C$ such that,
  \begin{equation*}
    \pr\paren{|d_v-\d|>(pn)^{2/3}}\leq \exp\paren[\Big]{-2C(pn)^{1/3}}.
  \end{equation*}
  Thus, by Markov's inequality, the probability that the number of
  vertices with degree outside $[\d-(pn)^{2/3},\d+(pn)^{2/3}]$ is more
  than $n\exp(-Cf^{1/3})$ is at most
  \begin{equation*}
    \frac{n \exp\paren[\big]{-2C(pn)^{1/3}}}
    {n\exp\paren[\big]{-Cf^{1/3}}}
    \leq
    \exp\paren[\big]{-C(pn)^{1/3}},
  \end{equation*}
  since $pn\geq f$.
 \end{proof}

\subsection{Maximum and minimum degree}
In this section, we collect several results about the maximum and minimum degree  of $\sG(n,p)$ relevant to our argument.
First, we give an easy upper-bound on the maximum degree.
\begin{lem}
\label{lem:maxdeg}
Given any constant $\gamma>0$, there exist positive constants $C$ and $K$ such that, if $p\le \gamma \log n/(n-1)$, then
the maximum degree of $\sG(n,p)$ is at most $K\log n$ with probability at least $1-n^{-C}$.
\end{lem}
\begin{proof}
  Let $\Delta$ denote the maximum degree in $G$ and $d_v$ denote the
  degree of $v$ in $G$ for any vertex~$v$. By union bound and
  Chernoff's bound in Theorem~\ref{thm:chernoff}, for any $K > \gamma$,
  \begin{equation*}
    \pr\paren{\Delta\geq K\log n}
    \leq
    n\pr\paren{d_v\geq K\log n}
    \leq
    \exp\paren[\bigg]{
      \frac{-(K\log n-p(n-1))^2}{3p(n-1)}
      +\log n
    }.
  \end{equation*}
  Since $pn\leq \gamma\log n$, it suffices to choose
  $K$ large enough so that $(K-\gamma)^2/(3\gamma) -1>0$.
\end{proof}
Our results about the minimum degree require the application of the first and the second moment methods
to the number of vertices of low degree.
The following lemma gives a lower-tail estimate for a Binomial random variable, and shall be used to bound the expected number of these low-degree vertices; the next lemma after that one will give us a bound on the variance.
 \begin{lem}\label{lem:binomial}
   For every constant $\eta>0$ there exist positive constants $C_1$ and $C_2$
   such that the following holds for any function $0\le p\le
   1/\sqrt{n}$ and every integer $0<k\le (1-\eta)np$. Let
   $X\sim \bin(n,p)$. Then,
\[
\pr(X\le k) = C \frac{e^{-pn}}{\sqrt k} \left(\frac{epn}{k}\right)^k
\quad
\text{with}
\quad
C_1 \leq C\leq C_2.
\]

\end{lem}
\begin{proof}
It follows easily from Stirling's
  approximation, that there exist two positive constants $A_1$ and $A_2$ such that, for every  $0<k<\sqrt n$,
  $$
  \frac{A_1}{\sqrt{k}}\left(\frac{en}{k}\right)^k \le \binom{n}{k}\le \frac{A_2}{\sqrt{k}}\left(\frac{en}{k}\right)^k.
  $$
Moreover, there exist positive constants $B_1$ and $B_2$ such that, for every $0\le p\le 1/\sqrt n$ and every $0<k<\sqrt n$,
\[
B_1 e^{-pn} \le (1-p)^{n-k} \le B_2 e^{-pn}.
\]
Therefore, there exist positive constants $C'_1$ and  $C'_2$ not depending on $p$ or $k$ such that
  \begin{equation*}
    \pr(X= k) =  \binom{n}{k} p^k (1-p)^{n-k} =  C' \frac{e^{-pn}}{\sqrt{k}}
    \paren[\Big]{\frac{epn}{k}}^k
\quad
\text{with}
\quad
C'_1 \leq C'\leq C'_2,
  \end{equation*}
and the lower bound follows immediately since $\pr(X\le k)\geq \pr(X= k)$.
For the upper bound, let $f_i=\pr(X=i)=\binom{n}{i}p^i (1-p)^{n-i}$, and observe that, for every $i\le (1-\eta)np$,
  $$
  \frac{f_{i-1}}{f_i}\le \frac{i}{(n-i)p} \le 1-\frac{\eta}{2},
  $$
since $p\leq 1/\sqrt n \le \frac{\eta}{2-\eta}$ eventually.
Hence, there is a constant $D>0$ only depending on $\eta$ such that
$\pr(X\le k)\leq D \pr(X=k)$.
\end{proof}

\begin{lem}
  \label{lem:secondm}
  Let $Y$ denote the number of vertices of degree at most $k$ in
  $\sG(n,p)$, where $p < 1$. Then
  $\var(Y) \leq (\ex Y)^2(p/(1-p)+1/\ex Y)$.
\end{lem}
\begin{proof}
  Recall that $d_v$ has distribution $\bin(n-1,p)$. Let $q_\leq(r,t)$
  denote the probability that a random variable with distribution
  $\bin(r,p)$ has value at most $t$ and let $q_=(r,t)$ denote the
  probability that it has value exactly~$t$. Then
  \begin{equation*}
    \label{eq:secondm_def}
    \begin{split}
      \ex\paren{Y^2}
      &=
      \sum_{u,v\in V}
      \pr\paren{d_v\leq k\text{ and }d_u\leq k}
      \\
      &=
      \ex{Y}
      +
      n(n-1)
      \left(
         p\cdot q_{\leq}(n-2,k-1)^2
         +
        (1-p)\cdot q_{\leq}(n-2,k)^2
      \right).
    \end{split}
  \end{equation*}
  This holds because for any distinct vertices $u,v\in V$, the number
  of neighbours of $u$ in $V\setminus\set{v}$ and the number of
  neighbours of $v$ in $V\setminus\set{u}$ are independent random
  variables with distribution $\bin(n-2,p)$. Clearly,
  $q_{\leq}(r,t)=q_{\leq}(r,t-1)+q_{=}(r,t)$. And so
  \begin{multline}
    p\cdot q_{\leq}(n-2,k-1)^2
    +
    (1-p)\cdot q_{\leq}(n-2,k)^2
    \\
    =
    q_{\leq}(n-2,k-1)^2
    +
    2(1-p)\cdot q_{=}(n-2,k)q_{\leq}(n-2,k-1)
    +
    (1-p)\cdot q_{=}(n-2,k)^2.
  \end{multline}
  Moreover,
  \begin{equation*}
    q_{\leq}(n-1,k)^2
    =
    \paren[\big]{q_{\leq}(n-2,k-1)+(1-p)\cdot q_{=}(n-2,k)}^2.
  \end{equation*}
  Thus,
  \begin{align*}
      \ex\paren{Y^2}
      &=
      \ex{Y}
      + n(n-1)
      \paren[\Big]{q_{\leq}(n-1,k)^2 + (1-p)q_{=}(n-2,k)^2(1-(1-p))}
      \\
      &\leq
      \ex{Y}
      + n^2 q_{\leq}(n-1,k)^2\left(1+\frac{p}{1-p}\right)
      = \ex(Y)^2\left(1+\frac{p}{1-p}+\frac{1}{\ex(Y)}\right).
\qedhere
  \end{align*}
\end{proof}
The following lemma bounds the probability that $\delta(\sG(n,p))$ deviates significantly from its expected value in the range $0.9\log n/(n-1) \leq p \leq \gamma\log n/(n-1)$, where $\gamma\ge0.9$ is a constant. We will apply this lemma when we require more precise probability bounds than those that would follow from the Chernoff's  inequalities in Theorem~\ref{thm:chernoff}.
\begin{lem}
  \label{lem:min_ranges_easy}
 Let $\gamma\ge0.9$ and $0<\eta<1$ be constants. Then there exists a constant $C>0$ such that, for any functions $p$ and $\alpha$ of $n$ satisfying $0<\alpha \le 1-\eta$ and $0.9\log n/(n-1) \leq p \leq \gamma\log n/(n-1)$, the following holds:
\begin{itemize}
\item[(i)]
$\displaystyle
\pr\paren[\Big]{\delta \le \alpha p(n-1)}
    \leq
     C \exp\paren[\bigg]{
        \log n - p(n-1)\paren[\Big]{1-\alpha\log\paren[\Big]{\frac{e}{\alpha}}}
         -\frac{1}{2}\log\log n
        }$ and

\item[(ii)]
$\displaystyle
    \pr\paren[\Big]{\delta > \alpha p(n-1)}
    \leq
    C\left( \frac{\log n}{n} + \exp\paren[\bigg]{
        p(n-1)\paren[\Big]{1-\alpha\log\paren[\Big]{\frac{e}{\alpha}}}-\log
        n+\frac{1}{2}\log\log n}\right)$.
\end{itemize}
\end{lem}
  \begin{proof}
    Given an arbitrary vertex $v$, let $d_v\sim\bin(n-1,p)$ be the
    degree of $v$. By Lemma~\ref{lem:binomial}, there exists a
    function $C'\in[C_1,C_2]$, where $C_1, C_2$ are positive constants
    that depend only on $\eta$ and $\gamma$  such that
  \begin{equation}
\label{eq:min-atleast-prob}
       \pr\paren{d_v\leq \alpha p(n-1)} =
      C' \exp\paren[\bigg]{
          -p(n-1)\paren[\Big]{1-\alpha\log\paren[\Big]{\frac{e}{\alpha}}}
           -\frac{1}{2}\log\log n}.
  \end{equation}
  Set $C$ to be a constant larger than $C_2 +2\gamma+1/C_1 \geq C'
  +2\gamma+1/C'$. From~\eqref{eq:min-atleast-prob}, the expected
  number of vertices with degree at most $\alpha p(n-1)$ is
\[
      C' \exp\paren[\bigg]{
         \log n - p(n-1)\paren[\Big]{1-\alpha\log\paren[\Big]{\frac{e}{\alpha}}}
           -\frac{1}{2}\log\log n},
\]
which implies~(i), since $C\ge C'$. Finally, the proof of~(ii) follows from Lemma~\ref{lem:secondm} and Chebyshev's inequality, since $p/(1-p)\le 2\gamma \log n/n$ and $C\ge 2\gamma + 1/C'$.
\end{proof}

It is convenient to state an easy consequence of Lemma~\ref{lem:min_ranges_easy} as a separate result.
We will use the following lemma when $p$ is very close to the threshold function $\beta\log n/(n-1)$ described in Theorem~\ref{thm:cases}, in order to have
a fairly precise bound of the probability that $\delta(\sG(n,p))$ deviates slightly from $pn/2$.
It is normally applied by choosing $\eps$ so that $|\eps|\log n$ is negligible compared to the other terms in~\eqref{eq:min-atleast} and~\eqref{eq:min-atmost}.

\begin{lem}
  \label{lem:min_ranges}
Let $\gamma>\beta=2/\log(e/2)$ and $0<\eta<1/2$ be constants. Then there exist positive constants $C$ and $D$ such
that the following holds. For any functions $p$ and $\eps$ of $n$ satisfying $|\eps| \le 1-2\eta$ and $0.9\log n/(n-1) \leq p \leq \gamma\log n/(n-1)$, we have
  \begin{eqnarray}
        \pr\paren[\Big]{\delta \le \frac{(1+\eps)}{2}p(n-1)}
    &\leq&
   C \exp\paren[\bigg]{
        -\frac{h}{\beta}-\frac{1}{2}\log\log n
        +D\abs{\eps}\log n
      }, \qquad\text{and}\label{eq:min-atleast} \\
      \pr\paren[\Big]{\delta > \frac{(1+\eps)}{2}p(n-1)}
    &\leq&
    C\left(\frac{\log n}{n}+\exp\paren[\bigg]{
        \frac{h}{\beta}+\frac{1}{2}\log\log n
        +D\abs{\eps}\log n}\right),\label{eq:min-atmost}
  \end{eqnarray}
where $h=h(n)$ is defined by
\[
p = \frac{\beta \log n + h}{n-1}.
\]
\end{lem}
\begin{proof}
By putting $\alpha=(1+\eps)/2$, we have that $\eta\le\alpha\le1-\eta$, and there is some constant $D'>0$ depending only on $\eta$ such that $|\alpha\log(e/\alpha) - 1 + 1/\beta|\le D' |\eps|$. The result follows immediately from Lemma~\ref{lem:min_ranges_easy} and setting $D=\gamma D'$.
\end{proof}
At this point, the reader may suspect that the relevant range of $p$ for the study of the evolution of $\delta(\sG(n,p))$ corresponds to $p=\Theta(\log n/n)$.
Indeed, a careful application of Lemma~\ref{lem:min_ranges_easy} yields the following: if $p\sim c\log n/n$ for some constant $c>1$, then a.a.s.\ $\d\sim c\log n$ and $\delta \sim g(c)\d$,
where $g:[1,\infty)\to(0,1)$ is a strictly increasing continuous function  with $\lim_{c\to1}g(c)=0$ and $\lim_{c\to\infty}g(c)=1$.
We do not prove the above claim, as it is not needed in our argument, but rather collect several related statements together in the following lemma.
\begin{lem}\label{lem:newMindeg}\hspace{0em}
\begin{enumerate}
\item[(i)] For any $p\le 0.9 \log n/(n-1)$, a.a.s.\ $\delta(\sG(n,p))=0$.
\item[(ii)]  For any constant $\epsilon>0$, there exist constants $\gamma>1$ and $C>0$
  such that, for every $\frac{0.9\log n}{n-1}\le p\le\frac{\gamma\log
    n}{n-1}$, we have that  $\delta(\sG(n,p))\le\epsilon\d(\sG(n,p))$ with probability at least $1-n^{-C}$. 
\item[(iii)]
   Let $\gamma > 1$ be a constant. There exist positive constants
   $\eps$ and $C$ such that, for $p\geq \gamma \log n/(n-1)$, we have
   that $\delta > \eps p(n-1)$ with probability at least $1-n^{-C}$. 
\item[(iv)] For every constants $0<\theta<1$ and $C>0$, there exists a constant $\gamma>0$, such that for all $p\ge \gamma\log n/(n-1)$,
we have $\pr(\delta(\sG(n,p))\le\theta p(n-1))\le n^{-C}$.
\end{enumerate}
\end{lem}
\begin{proof}
Part~(i) is a well-known fact (see~e.g.~\cite{bollobas85}).
We first prove part~(iv).  By Theorem~\ref{thm:chernoff},
 $$
 \pr\Big(\delta(\sG(n,p))\le \theta  p(n-1)\Big)\le n \exp\left(-(1-\theta)^2 p (n-1)/2\right)\le n \exp\left(-(1-\theta)^2(\gamma/2)\log n \right).
 $$
Thus, the statement holds by choosing $\gamma$ sufficiently large so that $(1-\theta)^2\gamma/2-1>C$.

Next, we prove part~(iii). From part (iv), there is a $\rho>1$ such that, for any $p\ge \rho \log n/(n-1)$,
$\delta(\sG(n,p))>(1/2)p(n-1)$ with probability at least $1-1/n$. Thus, for $\gamma\ge \rho$, statement (iii) follows immediately  by choosing any $\eps\le 1/2$ and any $C\le 1$.  So suppose otherwise that $1 < \gamma < \rho$. By Lemma~\ref{lem:min_ranges_easy},
  there is a positive constant $C'$ so that, for every $p$ in the considered range,
  \begin{equation*}
    \pr\paren{\delta(\sG(n,p))  \le \eps p(n-1)}
    \leq
    C'\exp\paren[\bigg]{
        \log n -\gamma \log n\paren[\Big]{1-\eps\log\paren[\Big]{\frac{e}{\eps}}}     }.
  \end{equation*}
  Since $\gamma > 1$, we can choose $0<\eps\le 1/2$ small enough so that
  $\gamma\paren{1-\eps\log\paren{\frac{e}{\eps}}} > 1$. Then there is a constant $C''>0$ such that the above probability is at most $n^{-C''}$. The statement follows by choosing $C=\min\{1,C''\}$.

Finally, we prove part~(ii).  We assume without loss of generality that $\eps < 1$.
   We have that
  \begin{equation}\label{eq:0}
    \pr\paren[\Big]{\delta(\sG(n,p)) > \eps \d(\sG(n,p))}
    \leq
    \pr\paren[\Big]{\d(\sG(n,p)) < 0.8\log n}
    +
    \pr\paren[\Big]{\delta(\sG(n,p)) > 0.8 \eps \log n}.
  \end{equation}
By Lemma~\ref{lem:avgdeg}, there is a positive constant $B$ such that, for any $p \ge 0.9\log n/(n-1)$,
  \begin{equation}
    \pr\paren[\Big]{\d(\sG(n,p)) < 0.8\log n}
    \leq
    \pr\paren[\Big]{\d(\sG(n,p)) < (8/9)pn}
    \leq
    \exp(-Bn\log n).\label{eq:1}
  \end{equation}
Let $1<\gamma < 8/7$ be a sufficiently small constant which we specify later, and put $\alpha=0.8/\gamma$.
In particular, $0.7<\alpha<0.8$.
By Lemma~\ref{lem:min_ranges_easy}, there is a constant $D>0$ such that, for every $p$ in the range
$0.9\log n/(n-1) \le p \le \gamma\log n/(n-1)$, we have
  \begin{eqnarray*}
   \pr\paren[\Big]{\delta(\sG(n,p)) > 0.8 \eps \log n}&\le& \pr\paren[\Big]{\delta(\sG(n,p)) > \eps \alpha p(n-1)}
    \\
    &\leq&
    D\left( \frac{\log n}{n}+ \exp\paren[\bigg]{
        \gamma\log n \paren[\Big]{1-\eps \alpha\log\paren[\Big]{\frac{e}{\eps\alpha}}}-\log
        n+\frac{1}{2}\log\log n}\right),
  \end{eqnarray*}
where we also used that $1-\eps\alpha\log(\frac{e}{\eps\alpha}) >0$, as $\eps\alpha<0.8$.
Moreover, choosing  $\gamma > 1$ small enough ensures that
\[
 B\eqdef \gamma\paren[\bigg]{1-\eps\alpha\log\paren[\Big]{\frac{e}{\eps\alpha}}} <
\gamma\paren[\bigg]{1-0.7\eps\log\paren[\Big]{\frac{e}{0.7\eps}}} < 1,
\]
  and the above probability is at most
  \begin{equation}
    D\left(\log n/n+ \exp\paren[\bigg]{
        -(1-B)\log n +\frac{1}{2}\log \log n}\right).\label{eq:2}
  \end{equation}
Combining~\eqref{eq:0}, \eqref{eq:1} and~\eqref{eq:2} yields statement~(ii), for $C$ sufficiently small.
\end{proof}
Finally, we include a result that compares the minimum degree of $\sG(n,p)$ and  $\sG(n,\hat p)$, when $p$ and $\hat p$ are close to one another.
%
%
%
\begin{lem}
  \label{lem:mindeg_near}
  For any constants $1<\gamma_1<\gamma_2$ and $\epsilon>0$, there exist positive constants
  $\eta$ and $C$ such that, for any functions $p$ and $\hat p$ satisfying
  $\gamma_1\log n/(n-1)\le p\le \hat p\le\gamma_2\log n/(n-1)$ and
  $\hat p/p-1\le\eta$,
  \[
   \frac{\delta(\sG(n,\hat p)}{\delta(\sG(n,p))}-1 \le
  \epsilon
  \]
  with probability at least
  $1-n^{-C}$.
\end{lem}
\begin{proof}
Assume without loss of generality that $0<\epsilon<1$. Choose constants $\gamma_0$ and $\gamma_3$ such that $1<\gamma_0<\gamma_1$ and $\gamma_2<\gamma_3$.
The function $f(y)=1-y\log(e/y)$ is a bijection from $[0,1]$ to $[0,1]$ (defining $f(0)=1$), and is strictly decreasing in that domain. Therefore, we can define the constants $y_i=f^{-1}(1/\gamma_i)$ for $i=0,1,2,3$, which satisfy $0<y_0<y_1<y_2<y_3<1$. Moreover, there exists a constant $D>0$ such that $f'\le-D$ for all $y$ in $[y_0, y_3]$, since this interval is a compact set and $f'<0$ is continuous there.

Pick two positive constants $\zeta$ and $\xi$ sufficiently small so that the following conditions are satisfied: $y_2+\zeta \le y_3$; $\zeta \le \eps y_0/3$; $\xi \le \gamma_1D\zeta/4$; and $\gamma_1/(1+\xi)\ge \gamma_0$ (note that the choice of $\xi$ depends on $\zeta$). With all these constants in mind, we choose $\eta>0$ in the statement sufficiently small so that $\eta \le \gamma_1D\zeta/4$ and $\eta \le \eps/3$.

Put $x=p(n-1)/\log n$ and $\hat x=\hat p(n-1)/\log n$. In general, $x$ and $\hat x$ are functions of $n$ with $\gamma_1\le x \le \hat x\le\gamma_2$, and moreover, from the assumption in the statement, $\hat x\le(1+\eta)x$.
Define $\alpha=f^{-1}((1+\xi)/x)$ and $\hat\alpha=\alpha+\zeta$, also functions of $n$. From the previous conditions $\gamma_1\le x \le\gamma_2$ and
$1+\xi\le \gamma_1/\gamma_0$, we deduce that $y_0\le\alpha\le y_2$. From this and since $y_2+\zeta \le y_3$, we get
$\hat\alpha\le y_3$. So in particular $\alpha,\hat\alpha\in[y_0,y_3]$. We have
\begin{equation}\label{eq:xfa}
x(1-\alpha\log(e/\alpha)) = xf(\alpha) = 1+\xi.
\end{equation}
Moreover, using the bound on $f'$ in $[y_0,y_3]$ and some of the earlier constraints on $x$, $\hat x$, $\xi$ and $\eta$,
\begin{equation}\label{eq:xfahat}
\hat x(1-\hat\alpha\log(e/\hat\alpha))\le (1+\eta)x(f(\alpha) - D\zeta) \le (1+\eta) ( 1+\xi - \gamma_1D\zeta )
 \le 1 - \gamma_1D\zeta/2.
\end{equation}
Using Lemma~\ref{lem:min_ranges_easy} together with~\eqref{eq:xfa} and~\eqref{eq:xfahat}, we conclude that
\[
\delta(\sG(n,p)) > \alpha x \log n
\qquad\text{and}\qquad
\delta(\sG(n,\hat p)) \le \hat\alpha \hat x \log n
\]
with probability at least $1-n^{-C}$, for any positive constant $C$ satisfying $C < \min\{1,\xi, \gamma_1D\zeta/2\}$. This last event implies that
\[
\delta(\sG(n,\hat p)) \le (1+\zeta/\alpha)(1+\eta) \alpha x \log n \le (1+\epsilon/3)^2 \delta(\sG(n,p)) \le (1+\epsilon)\delta(\sG(n,p)),
\]
since $\zeta \le \eps y_0/3 \le \eps \alpha/3$, $\eta \le \eps/3$ and $\eps<1$. This completes the proof of the Lemma.
\end{proof}
%
%
\subsection{Light vertices}

Recall that an $\eps$-light vertex was defined to be a vertex of degree at most $\delta+\eps\d$. The following result shows that a.a.s.\ all $\eps$-light vertices of $\sG(n,p)$ are at least three steps apart for a certain range of $p$.
\begin{lem}
  \label{lem:small}
  Suppose $0.9\log n/(n-1) \le p \le \gamma\log n/(n-1)$ for some
  constant $\gamma\ge0.9$. Then there exist constants $\epsilon>0$ and $C>0$
  such that the following holds in $\sG(n,p)$ with probability at
  least $1-n^{-C}$. There is no pair of adjacent $\eps$-light
  vertices and no two $\eps$-light vertices have a common neighbour.
\end{lem}

\begin{proof}
Let $x=p(n-1)/\log n$. For each $x\in[0.9,\gamma]$, define $\alpha=\alpha(x)$ to be the only solution in $(0,1)$ of
\begin{equation}
\label{eq:alpha}
x(1-\alpha\log(e/\alpha))=0.8.
\end{equation}
%
It is straightforward to verify that $\alpha\in(0,1)$ is well defined and strictly increasing with respect to $x\in[0.9,\gamma]$.
Consider the constant $\hat\eps = 0.1/(\gamma-0.8)$, and define $\hat\alpha=(1+\hat\eps)\alpha$. Recall that both $\alpha$ and $\hat\alpha$ are functions of  $x=p(n-1)/\log n$.
Then, using~\eqref{eq:alpha} and the fact that $\hat\eps \le 0.1/(x-0.8)$, we obtain
\begin{equation}
\label{eq:alphahat}
x(1-\hat\alpha\log(e/\hat\alpha)) > x(1 -\hat\alpha\log(e/\alpha)) = x - (1+\hat\eps)(x-0.8) \ge 0.7.
\end{equation}
{From}~\eqref{eq:alpha} and by Lemma~\ref{lem:min_ranges_easy}~(ii), we can bound
\begin{equation}
\label{eq:deltaalpha}
\pr(\delta>\alpha p(n-1))\le Dn^{-0.19},
\end{equation}
for a constant $D>0$ not depending on $p$. Assume for the rest of the argument that $D$ is sufficiently large.
Let $S$ be the set of vertices of degree at most $\hat\alpha p(n-1)$. By~\eqref{eq:min-atleast-prob} in the proof of Lemma~\ref{lem:min_ranges_easy} and~\eqref{eq:alphahat}, the probability that a vertex $v$ belongs to $S$ is
\begin{equation}
\label{eq:deltahatalpha}
\pr(v\in S)=\pr(d_v\le\hat\alpha p(n-1))\le Dn^{-0.7}.
\end{equation}
We can upper-bound the probability that a pair of vertices $u$ and $v$ are adjacent and belong to $S$, by
\[
p\pr(d_u\le\hat\alpha p(n-1))\pr(d_v\le\hat\alpha p(n-1))=p\paren[\big]{\pr(v\in S)}^2.
\]
Multiplying this by the number of possible pairs and using~\eqref{eq:deltahatalpha}, we get that the probability that $S$ contains some adjacent pair of vertices is at most
\begin{equation}\label{eq:Padj}
\binom{n}{2}p\paren[\big]{\pr(v\in S)}^2 \le D\gamma n^{-0.4} \log n.
\end{equation}
By a similar argument, the probability that $S$ contains a pair of vertices with a common neighbour is at most
\begin{equation}\label{eq:Pcomm}
\binom{n}{2}(n-2)p^2\paren[\big]{\pr(v\in S)}^2 \le D\gamma^2 n^{-0.4} \log^2 n.
\end{equation}
Finally, we define $\eps = \alpha(0.9) \hat\eps/2$. Recall that $\alpha$ is increasing in $[0.9,\gamma]$, and then
$\eps \le \hat \alpha \hat\eps/2$. It follows from Lemma~\ref{lem:avgdeg} that $\d=\d(\sG(n,p))$ is at most $2p(n-1)$
with probability at least $1-D/n$, assuming that $D$ is large enough.
If this event and the one in~\eqref{eq:deltaalpha} hold together, then
\[
\delta + \epsilon\d \le (\alpha+ 2\epsilon) p(n-1) \le \hat\alpha p(n-1),
\]
and therefore all $\eps$-light vertices are contained in $S$. Putting everything together, the statement holds with probability at least
$1 - n^{-C}$, for some small enough constant $C>0$.
\end{proof}
We include an extension of the previous lemma in terms of two random graphs $G_1\sim \sG(n,p_1)$ and $G_2\sim\sG(n,p_2)$, with
$0\le p_1\le p_2<1$, which are coupled together so that $G_1\subseteq G_2$.
This standard coupling can be achieved  in the following way. Let $G_1$ distributed as $\sG(n,p_1)$ and let $G_2$ the supergraph of $G_1$ obtained by adding each edge not in $G_1$ independently with probability $(p_2-p_1)/(1-p_1)$. Then $G_1\subseteq G_2$ and $G_2$ has the same distribution as $\sG(n,p_2)$ (for more details, we refer readers to Section~1.1 in~\cite{JLR}).
The following lemma will be used in Section~\ref{sec:process}, and can be proved in the exact same way as Lemma~\ref{lem:small}, but replacing $p$ by $p'$ in~\eqref{eq:Padj} and~\eqref{eq:Pcomm}.
\begin{lem}\label{lem2:small}
  Suppose $0.9\log n/(n-1) \le p \le p'\le\gamma\log n/(n-1)$ for some
  constant $\gamma\ge0.9$. Let $G_1\subseteq G_2$ where $G_1\sim \sG(n,p)$ and $G_2\sim \sG(n,p')$. Then there exist constants $\epsilon>0$ and $C>0$
  such that the following holds in $\sG(n,p)$ and $\sG(n,p')$ with probability at
  least $1-n^{-C}$. Let $S$ be the set of $\eps$-light
  vertices in $G_1$. Then in $G_2$, there is no edge induced by $S$, and no two vertices in $S$ adjacent to a common vertex.
\end{lem}

\subsection{Graph expansion}

For any sets $S,S'\subseteq [n]$, let $E(S,S')$ be the set of edges
in $G$ with one end in $S$ and the other in $S'$.

\begin{lem} \label{lem:expansion_large} Let $f\ge0$ be any function of
  $n$ such that $f\to\infty$, and $\zeta>0$ any fixed constant. Then,
  there exists a constant $C>0$ such that for every $f/n\le p\le 1$
  the following holds in $\sG(n,p)$ with probability at least
  $1-e^{-Cpn^2}$.  For every disjoint sets $S,S'\subseteq[n]$ with
  $\size{S},\size{S'}\ge \zeta n$ we have $|E(S,S')| \geq (\d/4)
  |S||S'|/n$.
\end{lem}
\begin{proof}
  The variable $|E(S,S')|$ has distribution $\bin(|S||S'|,p)$.  By
  Lemma~\ref{lem:avgdeg} with $\tau = 1/4$, we have that
  \begin{equation*}
  \pr\paren[\Big]{\d \geq \frac{5}{4}pn}\leq
  2\exp(-Apn^2)
  \end{equation*}
  where $A=1/12$ and $n\geq 2$. By Chernoff's bound in Theorem~\ref{thm:chernoff}, for a positive
  constant $B$,
  $$
  \pr\paren[\Big]{\bin\paren[\Big]{|S||S'|,p}< \frac{5p}{16} |S||S'|}
  \leq \exp\left(-Bp |S||S'|\right).
  $$
  Hence, the probability that there exist such $S$ and $S'$ is at most
  \begin{eqnarray*}
    \begin{split}
      &2\exp\paren[\Big]{-Apn^2}
      +
      \sum_{s,s'>\zeta n}
      \binom{n}{s}\binom{n}{s'}
      \exp\left(-B pss'\right)\le \exp\paren[\Big]{-Apn^2}
      +
      n^2\cdot 2^n\cdot 2^n
            \exp\left(-B \zeta^2pn^2\right)
      \\
      &\leq
      2\exp\paren[\Big]{-Apn^2}
      +
      \exp\left(
        B''n
        -B \zeta^2p n^2\right).
       \end{split}
    \end{eqnarray*}
    for a positive constant $B''$ and we are done since $pn\geq
    f\to\infty$ as $n\to\infty$.
\end{proof}

\begin{lem}
  \label{lem:sets}
  Let $f\ge0$ be any function of $n$ such that $f\to\infty$, and let
  $\alpha>0$ be any fixed constant. Then, there exist constants
  $\zeta>0$ and $C>0$ such that for every $f/n\le p\le 1$ the
  following holds in $\sG(n,p)$ with probability at least
  $1-Ce^{-(pn)^2}$.  For all $s\le\zeta n$ and every set $S$ of
  size~$s$, we have that $|E[S]| \leq \alpha p ns$.
\end{lem}
\begin{proof}
 The result is trivial for any set of size $s\leq 2\alpha pn$ since
 $\size{E[S]}\leq s^2/2 \leq s (2\alpha pn)/2 =
 \alpha pn s$. Let
  $\zeta>0$ be small enough so that
  $\frac{e\zeta}{2\alpha}<e^{-1/\alpha^2}$.  The expected number of
sets of size $2\alpha pn\le s\le \zeta n$ containing at least $\alpha pns$ edges is at
most
\begin{equation*}
  \begin{split}
    \binom{n}{s}\binom{\binom{s}{2}}{\ceil{\alpha pns}}p^{\ceil{\alpha
        pns}}
    &\le
    \left(\frac{en}{s} \left(\frac{es}{2\alpha n}\right)^{\ceil{\alpha
          pns}}\right)^s
    =
      \left(\frac{e^2}{2\alpha} \left(\frac{es}{2\alpha
            n}\right)^{\ceil{\alpha pns}-1}\right)^s
      \\
      &\leq
      \left(A\left(\frac{e\zeta}{2\alpha}\right)^{\alpha pn}\right)^s
      < \left(A e^{-pn/\alpha}\right)^s,
  \end{split}
  \end{equation*}
for some constant $A>0$ depending only on $f$ and $\zeta$ (we used the fact that the exponent $\ceil{\alpha pn}-1\ge \alpha f-1$, which eventually becomes positive as $f\to\infty$).

  Summing the expectation above over all $s\geq 2\alpha pn$, we get
\[
\sum_{s\geq 2\alpha pn} \left(A e^{-pn/\alpha}\right)^s \le \left(A e^{-pn/\alpha}\right)^{2\alpha pn} \frac{1}{1-A e^{-f/\alpha}}
\le C e^{-(pn)^2},
\]
for some constant $C>0$ only depending on $f$ and $\zeta$.
  \end{proof}

For any $S\subseteq [n]$,
let $\overline{S}$ denote $[n]\setminus S$.
\begin{lem}
  \label{lem:expansion_easy}
 Let $\gamma>1$ be a fixed constant. There exists a constant $C>0$ such that for any $p=p(n)\ge \gamma\log n/(n-1)$, the
  following holds in $\sG(n,p)$ with probability at least $1-n^{-C}$.  For every $S\subsetneq[n]$ with $2\le |S|\le n-2$,
  $|E(S,\overline{S})| \geq 1.5\delta$.
\end{lem}
\begin{proof} Without loss of generality, we may assume that $|S|\le |\overline{S}|$. Since $p\ge\gamma\log n/(n-1)$ for some $\gamma>1$, by Lemma~\ref{lem:newMindeg} (iii), there exist constants $\eps>0$ and $C_1>0$ such that with probability at least $1-n^{-C_1}$, $\delta=\delta(\sG(n,p))\ge \eps pn$. Let $\alpha=\eps/8$.
Then by Lemma~\ref{lem:sets}, there exist constants $\zeta>0$ and $C_2>0$ such that with probability at least $1-n^{-C_2}$, for all sets $S$ with size at most $\zeta n$, $|E(S,S)|\le \alpha pn|S|$.  Then, with probability at least $1-n^{-C_1}-n^{-C_2}$, for all these $S$, $|E(S,\overline{S})|\ge \delta|S|-2\alpha pn|S|\ge (3/4)\delta|S|\ge 1.5 \delta$, as $|S|\ge 2$. Now by Lemmas~\ref{lem:expansion_large} and~\ref{lem:avgdeg}, there exists another constant $C_3>0$ such that with probability at least $1-n^{-C_3}$, for all sets $S$ with size at least $\zeta n$, $|E(S,\overline{S})|\ge (\d/4)\zeta^2 n$, where $\d=\d(\sG(n,p))\sim np$. Clearly, $(\d/4)\zeta^2 n\ge 1.5\delta$ with probability at least $1-n^{-C_4}$ for some $C_4>0$. The lemma follows by choosing $C<\min\{C_i:\ 1\le i\le 4\}$.
\end{proof}
%
%
%
%
\section{Proof of Theorem~\ref {thm:main}}
\label{sec:main}

We proceed to prove Theorem~\ref {thm:main}, as a consequence of
Propositions~\ref{prop:a} and~\ref{prop:b}. For the rest of the argument, let
$\delta\eqdef \delta(\sG(n,p))$ and let $\d\eqdef \d(\sG(n,p))$. We
split the argument into cases depending on the range of $p$.

First observe that by Lemma~\ref{lem:newMindeg} (i) we can assume that $p\ge0.9\log
n/(n-1)$, since for $p\le0.9\log n/(n-1)$ the random graph $\sG(n,p)$ is a.a.s.\
disconnected and has minimum degree zero, so the statement of
Theorem~\ref {thm:main} holds trivially.

Let $\gamma_2$ be a large enough constant so that for
$p\ge\gamma_2\log n/(n-1)$ we have $\delta> (3/4)\d$ \aas\
(see Lemma~\ref{lem:newMindeg} (iv) and Lemma~\ref{lem:avgdeg}).  Let
$\eps<1/2$ be the constant given by Lemma~\ref{lem:small} with
$\gamma=\gamma_2$.  Let $\gamma_1 \in(1,\gamma_2)$ be the constant given
by Lemma~\ref{lem:newMindeg} (ii) with $\eps/4$.

For $0.9\log n/(n-1) \le p\le\gamma_1\log n/(n-1)$, we only need to
show that $\sG(n,p)$ a.a.s.\ satisfies the hypothesis of
Proposition~\ref{prop:a}. First, we note from Lemma~\ref{lem:avgdeg} that
$\d\sim pn\to\infty$. Condition (a) holds by our choice of
$\gamma_1$. Condition (b) follows from Lemma~\ref{lem:sets} with any
$\alpha<\epsilon/4$, since $\d\sim pn$.
Fix $\zeta$ as given by that lemma.
 Condition (c) with $\eta=\zeta^2/4$ is a consequence
of Lemma~\ref{lem:expansion_large}.

Finally, we show that $\sG(n,p)$ a.a.s.\ satisfies the conditions in
Proposition~\ref{prop:b} for the range $p\ge\gamma_1\log n/(n-1)$. First note
that $\delta = \Omega(\d)$ by Lemma~\ref{lem:newMindeg} (iii). Condition
(a') is satisfied for $p\ge\gamma_2\log n/(n-1)$, since \aas\
$\delta>\frac{(1+\epsilon)\d}{2}$ (by our choice of $\gamma_2$ and since $\eps<1/2$); and it is also satisfied for $\gamma_1\log
n/(n-1)\le p\le\gamma_2\log n/(n-1)$, since \aas\ no $\eps$-light vertices
are adjacent nor have a common neighbours (by our choice of
$\eps$).
For condition (e'), note that
$\epsilon t/\d$ is bounded away from $0$ since $\delta =
\Omega(\d)$. Therefore, the condition follows from Lemma~\ref{lem:sets}
with $\alpha = \epsilon t/(8\d)$ (also using that a.a.s.\ $\d\ge pn/2$), and this determines our choice of $\zeta$.
  Condition (b') holds \aas\ by
Lemma~\ref{lem:deg}. Condition (c') holds \aas\ by
Lemma~\ref{lem:expansion_large}. Condition (d') holds \aas\ by
Lemma~\ref{lem:expansion_easy}.  \qed

\section{Proof of Theorem~\ref {thm:cases}}
\label{sec:cases}

The number of edges in
$G\sim \sG(n,p)$ is a binomial random variable distributed as
$\bin(\binom{n}{2},p)$. If $p<0.9\log n/n$, then by Lemma~\ref{lem:newMindeg} (i), a.a.s.\ $\delta(G)=0$ and thus a.a.s.\ $\delta(G)\le \d(G)/2$. Assume $p\ge 0.9\log n/n$. By Lemma~\ref{lem:avgdeg}, a.a.s.\ $\abs{\d/2-pn/2}\leq \omega_n \sqrt{p}$, where $\d = \d(G)$. By Lemma~\ref{lem:newMindeg} (iv), there is a constant $\gamma>0$, such that for all $p\ge \gamma\log n/n$, a.a.s.\ $\delta(G)\ge (3/4)pn$. Hence, for $p$ in this range, a.a.s.\ $\delta(G)>\d/2$. Now we only consider $0.9\log n/n\le p\le \gamma\log n/n$.

Let $\w_n$ be a positive-valued function of $n$ that goes to infinity
arbitrarily slowly as $n\to\infty$, and let $\eps=\w_n/\sqrt{p}n$. 
Define $f=f(n)$ by $p=\frac{\beta\log n+f}{n-1}$. To prove statement~(ii), we assume $f\ge -\beta\log\log n/2+\w_n$.
 By Lemma~\ref{lem:min_ranges} (with $h=f$), we have that
\begin{equation}
  \pr\paren[\Big]{\delta \le \frac{1}{2}(1+\eps)pn}
    =
    O\paren[\Bigg]{\exp\paren[\bigg]{
        -\frac{f}{\beta}-\frac{1}{2}\log\log n
        +O(\eps\log n)
        }}
    =o(1),
  \end{equation}
  as $ -f/\beta -\log\log n/2 \le - w_n$, wheras $\eps\log n=O(\omega_n\sqrt{p})=o(1)$. Moreover, by Lemma~\ref{lem:avgdeg}, $\pr(\d/2\ge
  (1+\eps)pn/2)=o(1)$. Thus, a.a.s.\ $\delta > \d/2$ and so
  $T(G)=\floor{\d/2}$ by Theorem~\ref{thm:main}. This completes the proof of statement~(ii).
 On the other hand, if
  $f\le -\beta\log\log n/2 -\w_n$, then $f/\beta+\log\log n/2 \le -\w_n$, and thus by Lemma~\ref{lem:min_ranges} (with $h=f$, and $\eps$ replaced by $-\eps$),
\begin{equation}
  \pr\paren[\Big]{\delta > \frac{1}{2}(1-\eps)pn}
  =
  O\paren[\Bigg]{\frac{\log n}{n}+
    \exp\paren[\bigg]{
        \frac{f}{\beta}+\frac{1}{2}\log\log n
        +O(\eps\log n)
      }}
=o(1).
\end{equation}
Again by Lemma~\ref{lem:avgdeg},
$\pr(\d/2\le (1-\eps)pn/2)=o(1)$. Thus, a.a.s.\ $\delta<
\d/2$ and thus $T(G)=\delta(G)$ by Theorem~\ref{thm:main}, which completes the proof of statement~(i).
\qed

\section{Proof of Theorem~\ref{thm:process}}
\label{sec:process}

A standard tool to investigate the random graph process $G_0,\ldots,
G_m,\ldots, G_{\binom{n}{2}}$ is the related continuous random graph
process $(\sG_p)_{p\in[0,1]}$ defined as follows. For each edge $e$ of
the complete graph with vertex set $[n]$, we associate a random
variable $P_e$ uniformly distributed in $[0,1]$ and independent from
all others. Then, for any $p\in[0,1]$, we define $\sG_p$ to be the
graph with vertex set $[n]$ and precisely those edges $e$ such that
$p\ge P_e$. Note that for each $p$, $\sG_p$ is distributed as
$\sG(n,p)$. This provides us with a useful way of coupling together
$\sG(n,p)$ for several values of $p$, since $p\le p'$ implies $\sG_p
\subseteq \sG_{p'}$. Moreover, let $p(m)=\min\set{p\in[0,1]:\sG_p
  \text{ has at least $m$ edges}}$. Then, $\sG_{p(0)},\ldots,
\sG_{p(m)},\ldots, \sG_{p(\binom{n}{2})}$ is distributed as
$G_0,\ldots, G_m,\ldots, G_{\binom{n}{2}}$, since all $P_e$ are
different with probability $1$. For more details on the connection
between $(\sG_p)_{p\in[0,1]}$ and $(G_m)_{0\le m \le \binom{n}{2}}$
and further properties, we refer the reader to~\cite{JLR}.

In this article, we prove several statements that hold a.a.s.\
simultaneously for all $m$ in the random graph process $(G_m)_{0\le
  m\le\binom{n}{2}}$. To do so, it is often convenient to use small
bits of the continuous random graph process as follows. Given $p_0$
and $p_1$ as functions of $n$ such that $0\le p_0\le p_1\le 1$, we
consider $(\sG_p)_{p_0\le p\le p_1}$. Let $m_0=m(\sG_{p_0})$ and
$m_1=m(\sG_{p_1})$.  (Note that $m_0$ and $m_1$ are random variables
with $m_0\le m_1$, since $\sG_{p_0}\subset\sG_{p_1}$.) We colour all
edges of $\sG_{p_0}=G_{m_0}$ red and the remaining $m_1-m_0$ edges in
$\sG_{p_1}\setminus \sG_{p_0}$ blue. Then we can interpret
$G_{m_0},G_{m_0+1},\ldots,G_{m_1}$ as a random graph process in which
we sequentially add blue edges to $G_{m_0}$, so that each $G_m$ has
the $m_0$ red edges of $G_{m_0}$ together with the first $m-m_0$ blue
edges we add in the process. This interpretation will be used many
times throughout the argument.

We first prove the following result, which is stronger than part~(ii) of Theorem~\ref{thm:process}, and is also used in the argument for part~(i).
%
%
\begin{thm}\label{thm:continuousCases}
Consider the random graph process $(\sG_{p})_{0\le p\le 1}$. We have that a.a.s.
\begin{itemize}
\item[(i)] for all $p\le \frac{\beta (\log n-\log\log n/2)-\omega(1)}{n-1}$, we have $\delta(\sG_p)\le \d(\sG_p)/2$; and
\item[(ii)] for all $p\ge \frac{\beta (\log n-\log\log n/2)+\omega(1)}{n-1}$, we have $\delta(\sG_p)> \d(\sG_p)/2$.
\end{itemize}
Moreover, for every constant $0<\theta<1$, there is a constant $\rho>0$ such that a.a.s.
\begin{itemize}
\item[(iii)] for all $\rho \log n/(n-1)\le p\le 1$, we have $\delta(\sG_p) > \theta\d(\sG_p)$.
\end{itemize}
\end{thm}
\begin{proof} 
First, we prove statement (iii). We will show that for every $0<\theta<1$, there exists $\rho>0$ such that a.a.s.\
\begin{equation}
\delta(G_m)> \theta \d(G_m),\quad \mbox{ for all}\ m\ge m_0=(\rho/4) n\log n.\label{Gnm}
\end{equation}
 Then, let $p_0=\rho \log n/(n-1)$. By Chernoff's bound in Theorem~\ref{thm:chernoff}, a.a.s.\ $m(\sG(n,p_0)) > m_0$, i.e.\ a.a.s.\ $p(m_0)<p_0$. It follows then that a.a.s.\
$\delta(\sG_p)>\theta \d(\sG_p)$ for all $p\ge p_0$.
Now we prove~(\ref{Gnm}). For each $m$, let $\bar p=m/\binom{n}{2}$. Then
 $$
 \pr\Big(\delta(G_m)\le \theta\cdot 2m/(n-1) \Big)=\pr\Big(\delta(\sG(n,\bar p))\le \theta \bar pn \mid m(\sG(n,\bar p))=m) \Big).
 $$
 By the choice of $\bar p$, $h(i)=\pr(m(\sG(n,\bar p))=i)$ is maximized at $i=m$. Hence, $\pr(m(\sG(n,\bar p))=m)\ge n^{-2}$. Thus,
 $$
 \pr\Big(\delta(G_m)\le \theta\cdot 2m/(n-1) \Big)\le \frac{\pr\Big(\delta(\sG(n,\bar p))\le \theta \bar pn\Big)}{\pr(m(\sG(n,\bar p))=m)}\le n^2\pr\Big(\delta(\sG(n,\bar p))\le \theta \bar pn\Big).
 $$
 By Lemma~\ref{lem:newMindeg} (iv), for every $0<\theta<1$, we can choose $\rho>0$ sufficiently large such that
the probability on the right-hand side above is less than $1/n^5$ for every $m\ge (\rho/4) n\log n$ (correspondingly, $\bar p\ge (\rho/2)\log n/(n-1)$).
Hence, taking a union bound over the $O(n^2)$ possible values of $m$, we deduce that claim~(\ref{Gnm}) is true with probability at least $1-O(n^{-1})$.

Next, we prove statements (i) and (ii).
Let $f=o(\sqrt{\log n})$ be a function that goes to $\infty$ arbitrarily slowly, as $n\to\infty$.
Let $p_i=(\beta(\log n -\log\log n/2) -f^2-if)/(n-1)$ and let $q_i=(\beta(\log n-\log\log n/2) +f^2+if)/(n-1)$, for each $i\ge 1$.
Let $T$ be the smallest integer such that $p_T \le 0.9\log n/(n-1)$ and redefine $p_T=0.9\log n/(n-1)$. Let $\rho$ be the constant satisfying statement (iii) with $\theta=3/4$. Let $T'$ be the smallest integer such that $q_{T'} \ge \rho \log n/(n-1)$ and redefine $q_{T'}=\rho \log n/(n-1)$. Obviously, $T,T'=O(\log n)$.
\begin{claim}\label{claim:couple} There exists a positive constant
  $C$, such that, for every $1\le i< T$,
  $$
  \pr\Big(\delta(\sG_{p_i})>\d(\sG_{p_{i+1}})/2\Big)\leq C( f^{-i}+\log
  n/n),$$ and for very $1\le i< T'$,
  $$\pr\Big(\delta(\sG_{q_i})\le
  \d(\sG_{q_{i+1}})/2\Big)\leq C( f^{-i}+n^{-1}).$$
\end{claim}
By Lemma~\ref{lem:newMindeg} (i) and from the monotonicity of $\delta(\sG_p)$ with respect to $p$, a.a.s.\ for all $p<p_T=0.9\log n(n-1)$, we have $\delta(\sG_p)=0$ and thus $\delta(\sG_p)\le \d(\sG_p)/2$ holds.
 By Claim~\ref{claim:couple}, with probability at least
 $$
 1-\sum_{1\le i<T}C(f^{-i}+\log n/n)=1-o(1),
 $$
for all $1\le i< T$ and for all $p_{i+1}\le p\le p_i$,
$$
\delta(\sG_p)\le \delta(\sG_{p_i}) \le \d(\sG_{p_{i+1}})/2\le \d(\sG_p)/2.
$$
Thus, a.a.s.\ $\delta(\sG_p)\le \d(\sG_p)/2$ for all $p\le \frac{\beta
  (\log n -\log\log n/2)-\omega(1)}{n-1}$, since $f$ is an arbitrary slowly growing function in $\omega(1)$,
and statement (i) follows.

By~(iii) (with $\theta=3/4$), we only need to prove that a.a.s.\ $\delta(\sG_p)>\d(\sG_p)/2$ for all $p$ satisfying $q_1 \le p\le q_{T'}=\rho\log n/(n-1)$.
 Similarly as in the previous argument, with probability at least
$$
1-\sum_{1\le i<T'}C(f^{-i}+n^{-1})=1-o(1),
$$
for all $1\le i<T'$ and for every $p$ with $q_i\le p\le q_{i+1}$,
$$
\delta(\sG_p)\ge \delta(\sG_{q_i}) > \d(\sG_{q_{i+1}})/2\ge \d(\sG_p)/2.
$$
Thus, a.a.s.\ $\delta(\sG_p)>\d(\sG_p)/2$ for all $p\ge \frac{\beta (\log n -\log\log n/2)+\omega(1)}{n-1}$, as required in statement~(ii).

Finally, we prove Claim~\ref{claim:couple}.
In this argument the asymptotic statements are uniform for all $p\in[p_T,p_1]\cup[q_1,q_{T'}]$.
  By Lemma~\ref{lem:avgdeg}, for $\sigma=n^{-1/3}$ and a positive constant $A$,
\begin{equation}\label{eq:sad}
  \pr(|\d(\sG_{p})-p n|>\sigma p n) \le \exp(-A\sigma^2 n^2 p) = o (n^{-1}).
\end{equation}
Note that the event
$\delta(\sG_{p_i})> \frac{1-\sigma}{2}p_{i+1}n$ may be written as $\delta(\sG_{p_i})> \frac{1+\eps_i}{2}p_{i}(n-1)$ for some negative $\eps_i=-\Theta(f/\log n)$.
  Hence, using~\eqref{eq:sad} and also Lemma~\ref{lem:min_ranges} with $\eps=\eps_i$ and $h(n)=-\beta\log\log
  n/2-f^2-if$, we get that, for every $1\le i<T$,
\begin{eqnarray*}
\pr\left(\delta(\sG_{p_i})> \frac{\d(\sG_{p_{i+1}})}{2}\right)&\le& \pr\Big(\delta(\sG_{p_i})> \frac{1-\sigma}{2}p_{i+1}n\Big)+\pr\Big(\d(\sG_{p_{i+1}})<(1-\sigma)p_{i+1} n \Big)\\
&=& O\left(\frac{\log n}{n}+\exp\left(\frac{-f^2-if}{\beta}+O(f)\right)\right) + o(n^{-1}) =O(f^{-i}+\log n/n).
\end{eqnarray*}
Similarly,
we write
$\delta(\sG_{q_i}) \le \frac{1+\sigma}{2}q_{i+1}n$ as $\delta(\sG_{q_i}) \le \frac{1+\eps'_i}{2}q_{i}(n-1)$ for some $\eps'_i=\Theta(f/\log n)$.
Using again~\eqref{eq:sad} and Lemma~\ref{lem:min_ranges} with
$\eps=\eps'_i$ and $h(n)=-\beta\log\log n/2 + f^2+if$, we obtain,
for every $1\le i<T'$,
\begin{align*}
\pr\left(\delta(\sG_{q_i})\le \frac{\d(\sG_{q_{i+1}})}{2}\right)&\le \pr\Big(\delta(\sG_{q_i})\le \frac{1+\sigma}{2} q_{i+1}n\Big)+\pr\Big(\d(\sG_{q_{i+1}})>(1+\sigma) q_{i+1}n\Big)\\
&= O\left(\exp\left(\frac{-f^2-if}{\beta}+O(f)\right)\right) + o(n^{-1}) =O(f^{-i}+1/n).
\qedhere
\end{align*}
\end{proof}

\begin{proof} [Proof of Theorem~\ref{thm:process}]
We first prove statement~(ii).
Let $p_1=\frac{(1-\eps/2)\beta \log n}{n-1}$ and let $p_2=\frac{(1+\eps/2)\beta \log n}{n-1}$. For $i=1,2$, the number of edges in $\sG(n,p_i)$ is distributed as $\bin(\binom{n}{2},p_i)$. By Lemma~\ref{lem:avgdeg}, we have that a.a.s.\ $m(\sG(n,p_1))\ge  \frac{1-\eps}{1-\eps/2} p_1 \binom{n}{2}  = (1-\eps)\beta n\log n/2$, and $m(\sG(n,p_2))\le  \frac{1+\eps}{1+\eps/2} p_2 \binom{n}{2}  = (1+\eps) \beta n\log n/2$. Then, Theorem~\ref{thm:process} (ii) follows immediately from Theorem~\ref{thm:continuousCases}.

We now proceed to prove statement (i) of
Theorem~\ref{thm:process}. Recall that for any graph $G$,
$t(G)=\min\{\delta(G),\d(G)/2\}$. First, define $p_0=0.9\log n/(n-1)$,
$p_1=\gamma_1\log n/(n-1)$ and $p_2=\gamma_2\log n/(n-1)$, for some
constants $1<\gamma_1<\gamma_2$ that we specify later. We prove the
statement separately for $(\sG_p)_{p_0\le p\le p_1}$, $(\sG_p)_{p_1\le
  p\le p_2}$ and $(\sG_p)_{p_2\le p\le1}$. For $(\sG_p)_{0\le p\le
  p_0}$ it is trivially true since a.a.s.\ $\delta(\sG_p)=0$ for all
$0\le p\le p_0$, by Lemma~\ref{lem:newMindeg} (i) and the monotonicity of $\delta(\sG_p)$ with respect to $p$.

\paragraph{Part 1 (${p_0\le p\le p_1}$):}
Let $\epsilon>0$ be a constant chosen to satisfy Lemma~\ref{lem2:small} with  $\gamma=1.1$. Pick a sufficiently small constant $1<\gamma_1<1.1$ and recall $p_0=0.9\log n/(n-1)$ and $p_1=\gamma_1\log n/(n-1)$. From Lemma~\ref{lem:avgdeg}, a.a.s.
\begin{equation}\label{eq:dp01}
\d(\sG_{p_1})\le(4/3)\d(\sG_{p_0}).
\end{equation}
Moreover, in view of Lemma~\ref{lem:newMindeg}~(ii), we assume that $\gamma_1$ is small enough  so that a.a.s.
\begin{equation}\label{eq:deltap1dp0}
\delta(\sG_{p_1})\le(\eps/16)\d(\sG_{p_1})\le(\eps/12)\d(\sG_{p_0}).
\end{equation}
Colour edges in $\sG_{p_1}$ so that all edges in $\sG_{p_0}$ are
coloured red and all edges in $\sG_{p_1}\setminus \sG_{p_0}$ are
coloured blue, as described in the beginning of the section. For each
vertex $v\in \sG_{p_1}$, the red (blue) degree of $v$ is the number of
red (blue) edges incident with $v$.  Let $S$ be the set of
$\eps$-light vertices of $\sG_{p_0}$. Since $\delta(\sG_{p_0})=0$
a.a.s., the $\eps$-light vertices are a.a.s.\ precisely those vertices
with degree at most $\epsilon \d(\sG_{p_0})$ in $\sG_{p_0}$. By the
choice of $\epsilon$ and Lemma~\ref{lem2:small}, a.a.s.\ the vertices
in $S$ induce no edges and have no common neighbours in the
whole process $(\sG_p)_{p_0\le p\le p_1}$ as blue edges are added.

\remove{We claim that a.a.s.\ this remains true for $S$ in the
whole process $(\sG_p)_{p_0\le p\le p_1}$ as blue edges are added. It
is enough to verify this claim for $\sG_{p_1}$ once all the blue edges
have been added. The proof involves a similar computation to that in
the proof of Lemma~\ref{lem:small}: we change~\eqref{eq:small} to:
\begin{equation*}
  \begin{split}
    \pr_{\sG(n,p_1)}(d_v\le \eps\d(\sG(n,p_0)))
    &\leq
    \pr_{\sG(n,p_1)}(d_v\le 4\eps\gamma_1\log n)\\
    &\leq
    \exp\left(-p_1n(1-\alpha\log(e/\alpha))-\frac{1}{2}\log\log n+O(1)\right),
      \end{split}
\end{equation*}
and the rest of the argument holds for $\eps$, since $\eps$ was chosen
sufficiently small for all $p \leq \gamma \log n/(n-1)$ and
$\gamma > \gamma_1$. This proves the claim
about $S$.}

For each $p_0\le p\le p_1$, let $S_p$ be the set of
$(11\eps/16)$-light vertices of $\sG_p$ (i.e.\ vertices of
degree at most $\delta(\sG_p)+(11/16)\eps \d(\sG_p)$ in
$\sG_p$). Note that a.a.s.\ $S$ contains $S_p$
for all $p$ in this range, since for any $v\in S_p$, its degree in
$\sG_{p_0}$ (i.e.\ the red degree of $v$ in $\sG_{p}$) is at most
\[
\delta(\sG_p)+(11/16)\eps \d(\sG_p) \le
\delta(\sG_{p_1})+(11/16)\eps \d(\sG_{p_1}) \le
 (\eps/12)\d(\sG_{p_0})+(11/12)\eps \d(\sG_{p_0}) = \eps\d(\sG_{p_0}),
\]
where we used~\eqref{eq:dp01} and~\eqref{eq:deltap1dp0}.

We just showed that a.a.s.\ in $(\sG_p)_{p_0\le p\le p_1}$ the set of
$(11\eps/16)$-light vertices of $\sG_p$ induce no edges
and have no common neighbours. Moreover, from~\eqref{eq:deltap1dp0} and by
monotonicity of $\delta(\sG_p)$ and $\d(\sG_p)$ with respect to $p$,
we have that a.a.s.
\[
\delta(\sG_p)\le \delta(\sG_{p_1})\le
(\eps/12)\d(\sG_{p_0})\le(\eps/12)\d(\sG_p)
\]
in the whole process $(\sG_p)_{p_0\le p\le p_1}$.  Putting all that
together, we have that a.a.s.\ the conditions of Proposition~\ref{prop:a} are
satisfied in $(\sG_p)_{p_0\le p\le p_1}$ (replacing $\eps$ by
$(11/16)\eps$), and therefore a.a.s.\ $T(\sG_p)=\delta(\sG_p)$
simultaneously for all $p$ in this range.

\paragraph{Part 2 (${p_1\le
  p\le p_2}$):}
Recall that $p_1=\gamma_1\log n/(n-1)$ and $p_2=\gamma_2\log n/(n-1)$,
where $\gamma_1$ is as in Part 1, and $\gamma_2>\gamma_1$ is a
sufficiently large constant. In view of
Theorem~\ref{thm:continuousCases} (iii), we assume that $\gamma_2$ is
large enough so that a.a.s.\ $\delta(\sG_p)\ge(3/4)\d(\sG_p)$ in the
whole process $(\sG_p)_{p_2\le p\le 1}$.

Define $q_i=(1+1/\log n)^ip_1$ for each $i=0,1,2,\ldots$, and let $T$ be the smallest integer such that $q_T\ge p_2$. Redefine $q_T=p_2$. We have $T\le 2\log(\gamma_2/\gamma_1)\log n=O(\log n)$, since eventually
\[
(1+1/\log n)^{2\log(\gamma_2/\gamma_1)\log n} > \gamma_2/\gamma_1.
\]
To prove the statement for $(\sG_p)_{p_1\le p\le p_2}$, it suffices to see that for every $0\le i\le T-1$, we have $T(\sG_p)=\floor{t(\sG_p)}$ throughout the process $(\sG_p)_{q_i\le p\le q_{i+1}}$ with probability at least $1-1/\log^2n$, and then simply take a union bound over all $i$.

Let $\epsilon$ be as in Lemma~\ref{lem2:small} (putting $\gamma=\gamma_2$), and fix $0\le i\le T-1$. We verify that with probability at least $1-1/\log^2n$ all conditions (a')--(e') of Proposition~\ref{prop:b} are satisfied in $(\sG_p)_{q_i\le p\le q_{i+1}}$. We colour as before the edges of $\sG_{q_i}$ red, and the additional edges in $\sG_{q_{i+1}}\setminus\sG_{q_i}$ blue.

Let $S$ be the set of vertices that are $\eps$-light in $\sG_{q_i}$ (they have red degree at most $\delta(\sG_{q_i})+\eps \d(\sG_{q_i})$). For each $q_i\le p\le q_{i+1}$, define $S_p$ to be the set of vertices that are $\eps/2$-light in $\sG_p$.
From Lemma~\ref{lem:avgdeg} \remove{Lemma~\ref{lem:min_ranges_easy} (with constants $\eta$ and $\alpha$ such that $\gamma_2(1-\alpha\log(e/\alpha))<1$)}
 and Lemma~\ref{lem:mindeg_near}, we have that
\begin{equation}\label{eq:deltadq}
  \d(\sG_{q_i})\sim\d(\sG_{q_{i+1}})
  \qquad\text{and}\qquad
\delta(\sG_{q_{i+1}}) \le (1+\eps/3) \delta(\sG_{q_i})
\end{equation}
with probability at least $1-n^{-C}$, for some small enough constant $C>0$ not depending on $i$.
These equations imply that $S\supseteq S_p$ for all $p$ in our range, since the red degree of any vertex in $S_p$ is at most
\[
\delta(\sG_p)+(\eps/2) \d(\sG_p) \le \delta(\sG_{q_{i+1}})+(\eps/2) \d(\sG_{q_{i+1}})
 \le (1+\eps/3)\delta(\sG_{q_{i}})+(\eps/2 - o(1)) \d(\sG_{q_{i}})  \le \delta(\sG_{q_{i}})+\eps \d(\sG_{q_{i}}),
\]
where we also used the trivial fact that $  \delta(\sG_{q_{i}})\le\d(\sG_{q_i})$.
By Lemma~\ref{lem2:small},  with probability at least
$1-n^{-C}$ the vertices in $S$ do not get common neighbours or
induced edges as the blue edges are added in $(\sG_p)_{q_i\le p\le
  q_{i+1}}$. This implies condition (a') replacing $\eps$ by
$\eps/2$.
 By Lemma~\ref{lem:newMindeg} (iii), there
exists a constant $\sigma > 0$ such that, uniformly for all $p\in[p_1,p_2]$, we
have that $\delta(\sG_p) \geq \sigma pn$ with probability at least
$1-n^{-C}$.
Therefore, $t(\sG_{q_{i}})\geq \sigma' q_{i+1}n$
 with probability at least
$1-n^{-C}$ for a
positive constant $\sigma'$ not depending on $i$.
By applying Lemma~\ref{lem:sets} to $\sG_{q_{i+1}}$ with $\alpha<\eps \sigma'/4$, we deduce that condition
(e') holds with probability at least $1-n^{-C}$ during all the process $(\sG_p)_{q_i\le p\le q_{i+1}}$. This determines our choice of $\zeta$.
Next, observe that $\sG_{q_i}$ satisfies condition~(b') with probability at least $1-e^{-C(\log n)^{1/3}}$, by Lemma~\ref{lem:deg}, and also condition~(c') with probability at least $1-n^{-C}$, by Lemma~\ref{lem:expansion_large}. Therefore, in view of~\eqref{eq:deltadq}, both Conditions~(b') and~(c') hold simultaneously in all $(\sG_p)_{q_i\le p\le q_{i+1}}$ with probability at least $1-e^{-C(\log n)^{1/3}}$. Condition (d') holds trivially for sets $S$ of size $1$. For larger sets, 
Lemma~\ref{lem:expansion_easy} applied to $\sG_{q_i}$ together with~\eqref{eq:deltadq} imply that this condition holds
in all $(\sG_p)_{q_i\le p\le q_{i+1}}$ with probability $1-n^{-C}$.

Taking the union bound for all $0\le i\le T-1 =O(\log n)$, we have that a.a.s.\ $T(\sG_p)=\floor{t(\sG_p)}$ throughout the process $(\sG_p)_{p_1\le p\le p_2}$ by Proposition~\ref{prop:b}.

\paragraph{Part 3 (${p_2\le p\le1}$):}
Let $\gamma_2$ be as in Part 2, and $p_2=\gamma_2\log n/(n-1)$. Recall
from the definition of $\gamma_2$ that a.a.s.\
$\delta(\sG_p)\ge(3/4)\d(\sG_p)$ in the whole process $(\sG_p)_{p_2\le
  p\le 1}$, and therefore Condition~(a') in Proposition~\ref{prop:b} holds.
Define $q_i=(1+1/\log n)^ip_2$ for each $i=0,1,2,\ldots$, and let $T$
be the smallest integer such that $q_T\ge 1$. Redefine
$q_T=1$. Observe that $T\le 3\log^2n$, since eventually $(1+1/\log
n)^{3\log^2n}\ge n^{2.5}$. The same argument as in Part~2 shows that
for every $0\le i\le T-1$, Conditions (b')--(e') in Proposition~\ref{prop:b}
are satisfied throughout the process $(\sG_p)_{q_i\le p\le q_{i+1}}$
with probability at least $1-1/\log^3n$. Taking the union bound over
all $i$, we conclude that a.a.s.\ all condition in Proposition~\ref{prop:b}
hold and therefore $T(\sG_p)=\floor{\d(\sG_p)/2}$, during the whole process
$(\sG_p)_{p_2\le p\le 1}$.
\end{proof}


\section{Proof of Theorem~\ref {thm:arboricity-hitting}}
\label{sec:arboricity}

We first prove statement~(ii). In view of Theorem~\ref{thm:process}, we
assume that $T(G_m)=\min\{\delta(G_m),\lfloor m/(n-1)\rfloor\}$ for
all $m=0,1,\ldots,\binom{n}{2}$. Then we pick any $m$ such that
$\delta(G_m)\ge \d(G_m)/2=m/(n-1)$. If $n-1$ divides $m$, then
$T(G_m)=m/(n-1)$ and thus $A(G_m)=m/(n-1)$. If $n-1$ does not divide
$m$, let $m'$ be the smallest integer $m'>m$ divisible by
$n-1$. Since, the minimum degree is always an integer, we have
$\delta(G_m)\ge \lceil m/(n-1)\rceil = m'/(n-1)$. Moreover, $G_m$ is a
spanning subgraph of $G_{m'}$ and thus
$\delta(G_{m'})\ge\delta(G_m)\ge m'/(n-1)$. Therefore, from our
assumption on the random graph process, we have
$T(G_{m'})=m'/(n-1)=\lceil m/(n-1)\rceil$, and these $\lceil
m/(n-1)\rceil$ edge-disjoint spanning trees cover all edges of $G_m$,
so $A(G_m)=\lceil m/(n-1)\rceil$. This completes the proof of the
statement.

Next we proceed to prove statement~(i). Let $f$ be any function of $n$
such that $f\to\infty$ arbitrarily slowly and $f=o(\log n)$. Define
$p_j=(1+1/f)^jf/n$ for each $j=0,1,2,\ldots$, and let $T$ be the
largest integer such that $p_T\le\frac{\beta\log n}{(1-\epsilon/2)n}$.
\begin{claim}\label{claim:1} for every $0\le j<T$ and for every $m$
  such that $p_j\le p(m)\le p_{j+1}$, the bound $A(\sG_{p(m)})=A(G_m)
  \le \ceil[\big]{\frac{m+\phi_2}{n-1}}$ holds in the random graph
  process $(\sG_{p})_{p\in[p_j,p_{j+1}]}$ with probability at least
  $1-1/(p_jn)^2$.
\end{claim}
Assuming Claim~\ref{claim:1}, the probability that $A(G_m) \le \ceil[\big]{\frac{m+\phi_2}{n-1}}$ fails somewhere in the random graph process between $\sG(n,p_0)$ and $\sG(n,p_T)$ is at most
\begin{equation}\label{eq:fsum}
\frac{1}{f^2}\sum_{j=0}^{T-1} (1+1/f)^{-2j} \le \frac{1}{f^2}\sum_{j=0}^{\infty} (1+1/f)^{-j} = \frac{1+1/f}{f} = o(1).
\end{equation}
%


Moreover, we have that $p_T\ge\frac{\beta\log n}{(1-\epsilon/4)n}$
eventually (for $n> n_0$ depending only on $f$). Then there is $\sigma>0$ such that a.a.s.\ $m>(1+\sigma) \beta n\log n$ for every $m$ with $p(m)\ge p_T$. Then, by Theorem~\ref{thm:process} (ii), a.a.s.\ $\delta(G_m)> \d(G_m)/2$ for all $m$ with $p(m)\ge p_T$. Thus, we only need to restrict our discussion to $p_j$ with $j\le T$.

Similarly, let $T'$ be the largest integer such that $p_{T'}\le\frac{\beta\log n}{(1+\epsilon/2)n}$.
\begin{claim}\label{claim:2} for every $0\le j<T'$ and for every
  $m$ such that $p(m)\in [p_j,p_{j+1}]$, the bound $A(G_m) \ge
  \ceil[\big]{\frac{m+\phi_1}{n-1}}$ holds in the random graph process
  $(\sG_p)_{p\in [p_j,p_{j+1}]}$ with probability at least
  $1-1/(p_jn)^2$.
\end{claim}
By the same argument as in~\eqref{eq:fsum}, assuming Claim 2, the
probability that $A(G_m) \ge \ceil[\big]{\frac{m+\phi_1}{n-1}}$ fails
somewhere in the random graph process $(\sG_p)_{p\in [p_0,p_{T'}]}$ is
also $o(1)$.  Moreover, note that since $p_{T'}\ge\frac{\beta\log
  n}{(1+2\epsilon/3)n}$, then a.a.s.\ for every $m$ with
$p(m)\in[p_{T'},1]$, we have $2m/n\ge\frac{\beta\log
  n}{(1+3\epsilon/4)}$ by Lemma~\ref{lem:avgdeg}, and therefore
eventually $\phi_1\le1/2$ for all $m$ in this range. Hence, for all
$m$ not divisible by $n-1$, we eventually have
\[
A(G_m) \ge \ceil[\Big]{\frac{m}{n-1}} = \ceil[\Big]{\frac{m+1/2}{n-1}} = \ceil[\Big]{\frac{m+\phi_1}{n-1}}.
\]
Otherwise, for $m$ divisible by $n-1$, the condition $\delta(G_m) < \d(G_m)/2=m/(n-1)$ implies $A(G_m)>m/(n-1)$, since we cannot have a full factorisation of $G_m$  into $m/(n-1)$ spanning trees, so then
\[
A(G_m) \ge \frac{m}{n-1}+1 = \ceil[\Big]{\frac{m+1/2}{n-1}} = \ceil[\Big]{\frac{m+\phi_1}{n-1}}.
\]
Putting everything together, we showed that a.a.s.~\eqref{eq:AGPhi}
holds simultaneously for all $G_m$ in the random graph process
$\sG_{p}$ for $p$ between $f/n$ and $1$. Given any $m_0=\omega(n)$ as
in the statement, we may simply choose $f=m_0/n$. Then a.a.s.\
$m(\sG_{f/n})=m(\sG(n,f/n)) \le (3/4)fn  = (3/4)m_0$, and
statement (a) holds for the desired range of $m$. It only remains to
prove Claims~\ref{claim:1} and~\ref{claim:2}.

\begin{proof}[Proof of Claim~\ref{claim:1}]
  In this proof the asymptotic statements are uniform for all $p_j$
  and depend only on~$f$.  Given any $0\le j< T$, we consider the
  random graph process $(\sG_p)_{p\in [p_j,p_{j+1}]}$. Define
  $\delta_j=\delta(\sG_{p_j})$, $\d_j=\d(\sG_{p_j})$,
  $m_j=m(\sG_{p_j})$, $t_j=\min\{\delta_j,\d_j/2\}$.
  By~Lemma~\ref{lem:avgdeg}, we have $\d_j\sim \d_{j+1}\sim np_j$ with
  probability at least $1-Ce^{-p_jn}$ for a positive constant $C$.

Let $\hat\epsilon>0$ be a sufficiently small constant so that
\begin{equation}\label{eq:eps}
e^{-1}\left(\frac{2e}{1+2\hat\epsilon}\right)^{(1+2\hat\epsilon)/2}<e^{-(1-\epsilon)/\beta}
\qquad\text{and}\qquad
e^{-1}\left(\frac{2e}{1-\hat\epsilon}\right)^{(1-\hat\epsilon)/2}>e^{-(1+\epsilon/4)/\beta}.
\end{equation}
Colour the edges of $G_{m_j}=\sG_{p_j}$ red. Let
$G_{m_{j+1}}=\sG_{p_{j+1}}$. Colour edges in $G_{m_{j+1}}\setminus
G_{m_j}$ blue. For any vertex $v\in G_{m_{j+1}}$, define the red
(blue) degree of $v$ to be the number of red (blue) edges that are
incident with $v$. Call a vertex light if its red degree is at most
$\frac{(1+2\hat\epsilon)}{2}\d_{j+1}$. A vertex is called heavy if its
red degree is at least $(3/4)\d_{j+1}$. The vertices that are neither
light nor heavy are called medium vertices. We have that for any
constant $\alpha>0$, with probability $1-e^{-p_j n}$, by
Lemma~\ref{lem:avgdeg} and Lemma~\ref{lem:maxdeg},
\begin{equation}
\label{edges}
  \left|m_{j+1}-p_{j+1}\binom{n}{2}\right|<\alpha p_{j+1}n^2, \quad
\Delta(G_{m_{j+1}})=O(\log n),
\end{equation}
where $\Delta(G)$ denotes the maximum degree of $G$.
By Lemma~\ref{lem:binomial} with $k= \frac{1+2\hat\eps}{2} \d_{j+1}$,
the expected number of light vertices is at most
\begin{equation*}
  n\left(\frac{1}{e} \left(\frac{2e}{1+2\hat \eps}\right)^{\frac{1+2\hat\eps}{2}}+o(1) \right)^{pn}.
\end{equation*}
Thus by Markov's inequality and~\eqref{eq:eps}, the number of light
vertices is
\begin{equation}
\ell\le
\frac{n}{\d_{j+1}\exp\left(\frac{(1-\epsilon)}{\beta}\frac{2m_{j+1}}{n}\right)}=o(n),\label{lightVertices}
\end{equation}
with probability at least $1-e^{-D p_j n}$, for a positive constant
$D$.  Similarly, by Lemma~\ref{lem:binomial} with $k= \frac{3}{4} \d_{j+1}$,
with probability $1-e^{-Dp_j n}$ for a positive constant~$D$, the number of
heavy vertices is
\begin{equation}
h=n-o(n). \label{heavyVertices}
\end{equation}
For the following construction, assume
~\eqref{edges},~\eqref{lightVertices} and~\eqref{heavyVertices}
hold. We add $g=\d_{j+1}\ell$ new edges (different from the previous
red and blue edges) to $G_{m_{j+1}}$, which we colour green, in
such a way that every light vertex is incident with exactly $\d_{j+1}$
green edges; every heavy vertex is incident to at most one green edge;
and no green edge is incident to any medium vertices. (So green edges
only connect light and heavy vertices.) This can be done  by~\eqref{lightVertices}, since the
total number of green edges
\[
g \le \frac{n}{\exp\left(\frac{(1-\epsilon)}{\beta}\frac{2m_{j+1}}{n}\right)}
\]
is much smaller than the number of heavy vertices eventually. Finally, greedily
add $n$ yellow edges to $G_{m_{j+1}}$, different to all previous
red, blue and green ones, in a way that each yellow edge connects two
heavy vertices and each heavy vertex is incident with at most $3$
yellow edges (this can be done greedily since we have $h\sim n$ heavy
vertices by~\eqref{heavyVertices} and the maximum degree (adding red, blue and
green degrees together) is $O(\log n)$ by~\eqref{edges}).

We may regard the sequence of graphs $G_{m_j}\subset
G_{m_j+1},\cdots\subset G_{m_{j+1}}$ as a process in which we
sequentially add blue edges to $G_{m_j}$, so the edges of each $G_m$
are precisely the red ones together with the first $m-m_j$ blue
ones. For each $m$ in our range, we define $G'_m$ as $E(G'_m)=G_m\cup
E_g\cup E_y$, where $E_g$ is the set of green edges added to
$G_{m_{j+1}}$ and $E_y$ is an arbitrary subset of yellow edges added
to $G_{m_{j+1}}$ so that the number of edges of the resulting graph
$G'_m$ is a multiple of $n-1$. We now verify that
$G'_{m_j},\ldots,G'_{m_{j+1}}$ satisfy all conditions (a')--(e') of
Proposition~\ref{prop:b}, assuming that (\ref{edges}), (\ref{lightVertices})
and (\ref{heavyVertices}) and some additional events hold. We give bounds
on the probabilities of these events.

First observe that for all $m_j\le m\le m_{j+1}$, we have $\d(G'_m)\le
\frac{2(m_{j+1}+g+n)}{n-1} \le \d_{j+1}+2$ and similarly $\d(G'_m)\ge
\frac{2m_j}{n-1}=\d_j$, so
\begin{equation}
  \label{eq:deg_gprime}
  \d(G'_m)\sim \d_{j+1} \sim \d_j
\end{equation}
 Hence, for all $m$
in the range,
\begin{equation}\label{eq:deltaGpm}
\delta(G'_m)\ge \frac{(1+2\hat\epsilon)\d_{j+1}}{2}\ge \frac{(1+\hat\epsilon)\d(G'_m)}{2},
\end{equation}
so (a') holds and $\delta = \Omega(\d) = \omega(1)$. For any $S$, let
$d_r(S)$ denote the average red degree of $S$. By Lemma~\ref{lem:deg},
with probability $1-e^{-C(p_j n)^{1/3}}$ for a positive constant $C$,
for all $S$ with $S\ge \zeta n$, $d_r(S)\ge \d(G_{m_j})(1-o(1)) \geq
\d(G'_m)(1-o(1))$ by~\eqref{eq:deg_gprime}.  Thus, (b') holds by
noting that $d(S)\ge d_r(S)$. By applying
Lemma~\ref{lem:expansion_large} to the red edges and
using~\eqref{eq:deg_gprime}, we deduce condition (c') holds with
probablity at least $1-e^{-Cp_jn^2}$ for all $G_{m}'$.

For (e'), first note that $t(G'_{m_j})=\Omega(\d_{j+1})$. Then,
Lemma~\ref{lem:sets} applied to $G_{m_{j+1}}$ (i.e.\ only red and blue
edges) shows that with probability $1-Ce^{-(p_jn)^2}$ all sets $S$ of
size $s<\zeta n$ have at most $(\hat\epsilon/ 8 )t(G'_{m_{j}})s$ red and
blue edges inside.  Let $G''$ be obtained by adding all $g$ green
edges and all $n$ yellow edges. So $G'_m\subseteq G''$ for all $m_j\le
m\le m_{j+1}$. We bound the number of edges induced by $S$ in
$G''$. Let $s_2\le s$ be the number of heavy vertices in $S$. Each green edge induced by $S$ must be incident to one of the $s_2$ heavy
vertices inside, and the number of yellow edges induced by $S$ is at most
$3s_2$, since each heavy vertex is incident to at most $3$ of
them. Therefore, the number of green and yellow edges induced by $S$ is at
most $4s_2\le 4s$, and the total number of edges induced by $S$ in
  $G''$ (and thus, in all $G'_m$, $m_j\le m\le m_{j+1}$) is at most
$(\hat\epsilon/ 8 )t(G'_{m_{j}})s+4s\le (\hat\epsilon/ 4 )
t(G'_{m_{j}})s$, since $t(G'_{m_{j}})\to\infty$.
Thus (e') holds for
all $m_j\le m\le m_{j+1}$, since $t(G'_{m_j}) \le t(G'_m)$.

Finally, we prove (d'): Let $S$ be a set with $1\le|S|\le n/2$
(otherwise we take $\overline S$). Suppose first that $1\le |S|\le
\zeta n$. From what we proved before for (e') and~\eqref{eq:deltaGpm},
the number of edges induced by $S$ is at most $(\hat\epsilon/ 4 )
t(G'_m)s$ and so, for each $G'_m$,
\[
E(S,\overline S)\ge \delta(G'_m)s- (\hat\epsilon/ 4 ) t(G'_m)s \ge \frac{(1+\hat\epsilon)\d(G'_{m_j})}{2}s- (\hat\epsilon/ 4 ) t(G'_{m_j})s \ge t(G'_{m_j})s \ge t(G'_{m_j}).
\]
Otherwise, if $\zeta n\le |S|\le n/2$, (c') gives us what we need using only red edges.

Hence, in view of Proposition~\ref{prop:b}, with probability $1-e^{-C(p_j n)^{1/3}}
\geq 1-1/(p_j n)^2$ since $p_j n\ge f=\omega(1)$, for all $m_j\le m\le m_{j+1}$, we have
\[
T(G'_m)= m(G'_m)/(n-1) = \ceil{(m+g)/(n-1)} \le \ceil{(m+\phi_2)/(n-1)},
\]
 since by construction  $m(G'_m)$ is the smallest integer that is at least $m+g$ and is divisible by $n-1$.

This implies the claim since
\[
A(G_m)\le A(G'_m)=m(G'_m)/(n-1) \le \ceil{(m+\phi_2)/(n-1)}.
\qedhere
\]
\end{proof}

\begin{proof}[Proof of Claim~\ref{claim:2}]
  We pick a constant $\hat\epsilon>0$ as
  in~\eqref{eq:eps}. By~(\ref{edges}) and the definition of $p_j$,
    with probability $1-e^{-p_jn}$, for all
  $m_j\le m\le m_{j+1}$ we have
\begin{equation}
A(G_m)\ge \lceil m/(n-1)\rceil \ge  \d(G_{m_j})/2 > (1-\hat\epsilon)\d(G_{m_{j+1}})/2\label{AGm}
\end{equation}
and the maximum degree of $G_m$ is $O(\log n)$. Let us redefine light
vertices of $G_{m_{j+1}}$ to be vertices with degree (red degree plus
blue degree) at most $(1-\hat\epsilon)\d(G_{m_{j+1}})/2$. By
Lemma~\ref{lem:binomial} with $k =
\frac{(1-\hat\eps)}{2}\d(G_{m_{j+1}})$ and~\eqref{eq:eps}, the
expected number of light vertices in $G_{m_{j+1}}$ is at least
\begin{equation*}
  n\sqrt{\frac{1}{k}}
  \left( \frac{1}{e}\left(\frac{2 e}{(1-\hat
        \eps)}\right)^{\frac{1-\hat\eps}{2}}+o(1)
  \right)^{pn}
  \geq
 n e^{-(1+\eps'/4)pn/\beta}
  \geq
 2 n e^{-\frac{(1+\eps/4)}{\beta}\frac{2m_j}{n}},
\end{equation*}
for some constant $\eps' < \eps$ and the inequality holds since $f =
o(\log n)$.
Note that with probability at least $1-n^{-C}$, for all $j\le T'$, $2m_j\le (1+\eps/8)\beta n\log n$ and so
$$
\frac{(1+\eps/4)}{\beta}\frac{2m_j}{n}\le \frac{(1+\eps/4)(1+\eps/8)}{1+\eps/2}\log n \le \sigma\log n,
$$
for some $0<\sigma<1$, depending only on $\eps$.
 Then, by Lemma~\ref{lem:secondm} and by Chebyshev's inequality, there
are
 \[
\ell' \ge \frac{n}{\exp\left(\frac{(1+\epsilon/4)}{\beta}\frac{2m_j}{n}\right)}
\]
light vertices in $G_{m_{j+1}}$ with probability at least $1-O(\log
n/n)-n^{\sigma-1}\geq 1-1/(p_jn)^2$. By~(\ref{AGm}), with probability
at least $1-e^{-p_jn}$, these light vertices have degree in $G_m$
strictly less than $A(G_m)$ for all $m$ in the range $m_j\le m\le
m_{j+1}$ since $G_m\subseteq G_{m_{j+1}}$.

For each $G_m$ we
construct $G'_m$ as follows. Let $\mathcal F_m$ be the set of $A(G_m)$
edge-disjoint forests covering $G_m$. For every light vertex $v$ and
for every forest $F\in\mathcal F_m$, if $F$ has no edge incident to
$v$, then we add a new edge connecting $v$ to some non-light vertex,
and make this new edge be part of $F$ (this can always be done since
both $|\mathcal F_m|$ and the maximum degree are $O(\log n)$ and the
number of non-light vertices is $n-o(n)$). Observe that $G'_m$ has at
least $m+\ell'$ edges since for each light vertex $v$ we added at
least one edge as the degree of $v$ is less than $|\mathcal
F_m|=A(G_m)$. By construction, $G'_m$ and $G_m$ have the same
arboricity, so
\[
A(G_m)=A(G'_m) \ge \lceil(m+\ell')/(n-1)\rceil \ge \lceil(m+\phi_1)/(n-1)\rceil.
\qedhere
\]
\end{proof}


\section{Proof of Theorem~\ref{thm:arboricity}}
\label{sec:sparse}

\begin{lem}\label{lem:smallS} Let $G\sim \sG(n,p)$, where $p\le c/n$ for some constant $c>0$. Then there exists another constant $\alpha>0$, such that a.a.s.\ all subgraphs of $G$ with order at most $\alpha n$ have average degree at most $2.2$.
\end{lem}
\begin{proof}Let $X_s$ denote the number of sets $S\subseteq[n]$ with $|S|=s$ and $|E[S]|> 1.1s$. Let $r=s/n$. Then
$$
\ex X_s\le \binom{n}{s}\binom{s^2}{1.1s}\left(\frac{c}{n}\right)^{1.1s}\le \left(\frac{e}{r}\left(\frac{ern}{1.1}\frac{c}{n}\right)^{1.1}\right)^{s}=\left(Cr^{0.1}\right)^s,
$$
where $C=e^{2.1}c^{1.1}/1.1^{1.1}$ is a constant depending only on $c$. Thus, by choosing $\alpha$ sufficiently small, we have that for all $r\le\alpha$, $Cr^{0.1}<1/2$.
It follows then that
\[
\sum_{1\le s\le \alpha n}\ex X_s=\sum_{1\le s\le \log n} \ex X_s+\sum_{\log n< s\le \alpha n} \ex X_s= O\left(\frac{\log^{0.1} n}{n^{0.1}}+2^{-\log n}\right)=o(1).
\qedhere
\]
\end{proof}

For any integer $k\ge 0$ and real $\mu\ge 0$, define
$$
f_k(\mu)=e^{-\mu}\sum_{i\ge k}\frac{\mu^i}{i!}.
$$
For any $k\ge 3$, define $
h_{k}(\mu)=\frac{\mu}{f_{k-1}(\mu)}$,
and let
\begin{equation}
c_{k}=\inf\{h_{k}(\mu),\mu>0\},\ \forall k\ge 3\quad \mbox{and}\ c_2=1. \label{c}
\end{equation}
 For any $c>c_{k}$, define $\mu_{c,k}$ to be the larger solution of $h_{k}(\mu)=c$.

The following theorem follows from a result about the threshold for the appearance of a giant $k$-core, first proved by Pittel, Spencer and Wormald~\cite{PSW}, and later  re-proved by many authors. See~\cite{Kim, Janson, Molloy}.
\begin{thm}\label{thm:coreThreshold}
Let $k \ge 2$ be fixed and let $c_k$ be defined as in~\eqref{c}. Then for all
$c>c_{k}$, 
 a.a.s.\ $\sG(n,c/n)$ has a non-empty $k$-core with $f_k(\mu_{c,k}) n+o(n)$ vertices and $\frac{1}{2}\mu_{c,k}f_{k-1}(\mu_{c,k})n+o(n)$ edges. For all $k\ge 3$ and $c<c_k$, a.a.s.\ $\sG(n,c/n)$ has an empty $k$-core.
\end{thm}

Cain, Sanders and Wormald~\cite{CSW} proved that for every $k\ge 2$ and $\epsilon>0$, a.a.s.\ we have that if the average degree of the $(k+1)$-core of $\sG(n,p)$ is at most $2k-\eps$, then $\sG(n,p)$ is $k$-orientable: all its edges can be oriented so that no vertex has indegree more than $k$. On the other hand, Hakimi's characterisation~\cite{Hakimi} tells that a graph is $k$-orientable if and only if it contains no subgraph whose average degree is more than $2k$. These two results immediately imply the following theorem.

\begin{thm}\label{thm:orientablity}
Given any positive integer $k\ge 2$ and an arbitrarily small $\eps>0$, a.a.s.\ we have that if the average degree of the $(k+1)$-core of $\sG(n,p)$ is at most $2k-\eps$, then there is no subgraph of $\sG(n,p)$ whose average degree is more than $2k$.
\end{thm}

\begin{cor}~\label{cor:subgraph}
Given any positive integer $k\ge 2$, a.a.s.\ we have that if the average degree of the $(k+1)$-core of $\sG(n,p)$ is at most $2k+o(1)$, then there is no subgraph of $\sG(n,p)$ whose average degree is more than $2k+o(1)$.

\end{cor}
\begin{proof} It is easy to verify that $\mu f_{k-1}(\mu)/f_k(\mu)$ is a strictly increasing function of $\mu$, which goes to infinity as $\mu\to\infty$. It is also easy to verify that $h_k(\mu)$ is a strictly increasing function of $\mu$ for $\mu\ge \mu_{c_k,k}$. Hence, by Theorem~\ref{thm:coreThreshold}, there exists a constant $c>0$ such that a.a.s.\ the average degree of the $(k+1)$-core of $\sG(n,c/n)$ is $2k+o(1)$.

Let $\eps>0$ be an arbitrarily small constant. For any $p\ge (c+\eps)/n$, a.a.s.\ the average degree of $\sG(n,p)$ is greater than $2k+\sigma$, for some $\sigma(\eps)>0$, by Theorem~\ref{thm:coreThreshold}. Hence, we may assume that $p\le (c+\eps)/n$, and it follows by Theorem~\ref{thm:coreThreshold} that the $(k+1)$-core of $\sG(n,p)$ has average degree at most $2k+O(\eps)$.
We will prove that for every $\eps>0$ and every $p\le (c+\eps)/n$, a.a.s.\ all subgraphs of $\sG(n,p)$ have average degree at most $2k+O(\epsilon)$.
 Construct $\sG(n,p)$ by exposing each non-edge in $\sG(n,(c-\eps)/n)$ independently with probability $p'$, where $p'$ satisfies
$$
\frac{c-\eps}{n}+\left(1-\frac{c-\eps}{n}\right)p'=p\le\frac{c+\eps}{n}.
$$
Then $p'=O(\eps/n)$. Thus, a.a.s.\ $\sG(n,p)$ contains $O(\eps n)$ extra more edges than $\sG(n,(c-\eps)/n)$. Let $H$ be a densest subgraph of $\sG(n,p)$. If the average degree of $H$ is less than $3\le 2k$, we are done. Otherwise, by Lemma~\ref{lem:smallS}, a.a.s.\ $|V(H)|=\Omega(n)$.
 By Theorem~\ref{thm:coreThreshold}, there exists $\sigma'>0$ depending on $\eps$ such that the $(k+1)$-core of $\sG(n,(c-\eps)/n)$ has average degree at most $2k-\sigma'$.
By Theorem~\ref{thm:orientablity}, a.a.s.\ the subgraph of $\sG(n,(c-\eps)/n)$ induced by $V(H)$ has average degree at most $2k$. Adding $O(\eps n)$ edges to $V(H)$ will change its average degree by $O(\eps)$ because $|V(H)|=\Omega(n)$. Thus, the average degree of $H$ in $\sG(n,p)$ is at most $2k+O(\eps)$. This holds for every $\eps>0$. Hence, a.a.s.\ there is no subgraph of $\sG(n,p)$ whose average degree is more than $2k+o(1)$ for every $p\sim c/n$.
\end{proof}

\begin{proof}[Proof of Theorem~\ref{thm:arboricity}]
Part (i) follows as a corollary of Theorem~\ref{thm:arboricity-hitting}. Now we consider $p=\Theta(1/n)$. Assume $p\le c/n$ for some constant $c>0$. Then there exists $\sigma_c>0$ such that a.a.s.\ the number of vertices in $G\sim\sG(n,p)$ with degree $0$ is at least $\sigma_c n$, whereas a.a.s.\ $m(G)=(1+o(1))cn/2$. Hence, every forest contained in $G$ has at most $(1-\sigma_c)n$ edges. It follows then that $A(G)\ge m(G)/(1-\sigma_c)n\ge (1+\Theta(1))c/2$. 
Next, we prove that for all $c=c_n=\Theta(1)$, a.a.s.\ $A(\sG(n,c/n))$ is concentrated on two bounded values. This directly implies that $A(\sG(n,c/n))$ is bounded and thus, for every $c=\Theta(1)$, a.a.s.\ $A(\sG(n,c/n))=(1+\Theta(1))c/2$.

To prove that $A(\sG(n,c/n))$ is concentrated on two values,
we consider two cases.
If $\limsup_{n\to\infty}c\le 1$, then all vertices of $\sG(n,c/n)$ are contained in isolated trees except a set $S$ of $o(n)$ vertices~\cite[Theorem 4b]{ER}. By Lemma~\ref{lem:smallS}, a.a.s.\ any subset $S'\subseteq S$ contains at most $1.1|S'|$ edges. Hence, a.a.s.\ there is no subgraph $H$ of $\sG(n,c/n)$, such that $|E(H)|/(|H|-1)>2$. It follows then that $A(\sG(n,c/n))\in\{1,2\}$ in this case.
Now we assume that $\liminf_{n\to\infty} c>1$. We prove the a.a.s.\ two-value concentration of $A(\sG(n,p)$ with $p\sim c/n$ for some constant $c>1$ and our claim holds for all $p$ in this rage by the subsubsequence principle (see~\cite{JLR}). By Theorem~\ref{thm:coreThreshold}, for every $k$, such that $c_k<c$, a.a.s.\ the average degree of the $k$-core of $A(\sG(n,c/n))$ is $\mu_{c,k}f_{k-1}(\mu_{c,k})/f_k(\mu_{c,k})+o(1)$. Let
$$
K_c=\{k:\ c_k\le c,\ \mu_{c,k}f_{k-1}(\mu_{c,k})/f_k(\mu_{c,k})>2(k-1)\}.
$$
Since $c>c_2=1$, a.a.s.\ there is a giant $2$-core whose average degree is strictly greater than $2$, and thus trivially $2\in K_c$. So $K_c$ is non-empty. It is well known that $c_k$ is an increasing sequence of $k$ and $c_k\to\infty$ as $k\to\infty$. Hence, $K_c$ is finite. Let $k_c$ be the largest integer in $K_c$.
 By Theorem~\ref{thm:coreThreshold}, a.a.s.\ the $(k_c+1)$-core has average degree at most $2k_c+o(1)$ and the $k_c$-core has average degree strictly greater than $2(k_c-1)$.
By Corollary~\ref{cor:subgraph}, a.a.s.\ there is no subgraph of $\sG(n,c/n)$ whose average degree is more than $2k_c+o(1)$. Hence, the average degree of the densest subgraph is a.a.s.\ strictly greater than $2(k_c-1)$ and at most $2k_c+o(1)$. It follows immediately by Theorem~\ref{thm:nw-arb}  that $A(\sG(n,c/n))\in\{k_c,k_c+1\}$ for all $c=\Theta(1)$ such that $\liminf_{n\to\infty} c>1$.
Part (b) follows by defining $k_c=1$ for all $c$ such that $\limsup_{n\to\infty}c\le 1$.

When $p=o(1/n)$, we have that $\sG(n,p)$ is a.a.s.\ acyclic by~\cite{ER}. So a.a.s.\ $A(\sG(n,p))\le1$.
\end{proof}

\bibliographystyle{plain}
\bibliography{treepack}

\begin{thebibliography}{10}

\bibitem{Azar99}
Y.~Azar, A.~Z. Broder, A.~R. Karlin, and E.~Upfal.
\newblock Balanced allocations.
\newblock {\em SIAM journal on computing}, 29(1):180--200, 1999.

\bibitem{Bollobas01}
B.~Bollob{\'a}s.
\newblock {\em Random Graphs}.
\newblock Cambridge University Press, 2001.

\bibitem{bollobas85}
B.~Bollob{\'a}s and A.~Thomason.
\newblock Random graphs of small order.
\newblock In {\em Random graphs}, volume~83, pages 47--97, 1985.

\bibitem{CSW}
J.~A. Cain, P.~Sanders, and N.~Wormald.
\newblock The random graph threshold for {$k$}-orientability and a fast
  algorithm for optimal multiple-choice allocation.
\newblock In {\em Proceedings of the {E}ighteenth {A}nnual {ACM}-{SIAM}
  {S}ymposium on {D}iscrete {A}lgorithms (SODA)}, pages 469--476, 2007.

\bibitem{CatlinChenPalmer93}
P.~A. Catlin, Zhi-Hong Chen, and E.~M. Palmer.
\newblock On the edge arboricity of a random graph.
\newblock {\em Ars Combin.}, 35(A):129--134, 1993.

\bibitem{ChenLiLian13}
X.~Chen, X.~Li, and H.~Lian.
\newblock Note on packing of edge-disjoint spanning trees in sparse random
  graphs.
\newblock http://arxiv.org/abs/1301.1097v1.

\bibitem{Chiba85}
N.~Chiba and T.~Nishizeki.
\newblock Arboricity and subgraph listing algorithms.
\newblock {\em SIAM Journal on Computing}, 14(1):210--223, 1985.

\bibitem{Cunningham85}
W.~H. Cunningham.
\newblock Optimal attach and reinforcement of a network.
\newblock {\em Journal of the ACM}, 32(3):549--561, 1985.

\bibitem{DGMMPR}
M.~Dietzfelbinger, A.~Goerdt, M.~Mitzenmacher, A.~Montanari, R.~Pagh, and
  M.~Rink.
\newblock Tight thresholds for cuckoo hashing via {XORSAT}.
\newblock http://arxiv.org/abs/0912.0287.

\bibitem{ER}
P.~Erd{\H{o}}s and A.~R{\'e}nyi.
\newblock On the evolution of random graphs.
\newblock {\em Bull. Inst. Internat. Statist.}, 38:343--347, 1961.

\bibitem{FR}
D.~Fernholz and V.~Ramachandran.
\newblock The $k$-orientability thresholds for {$G_{n,p}$}.
\newblock In {\em Proc. ACM-SIAM Symposium on Discrete Algorithms (SODA)},
  pages 459 -- 468, January 2007.

\bibitem{FMP}
N.~Fountoulakis, M.~Khosla, and K.~Panagiotou.
\newblock The multiple-orientability thresholds for random hypergraphs.
\newblock In {\em Proceedings of the Twenty-Second Annual ACM-SIAM Symposium on
  Discrete Algorithms (SODA)}, pages 1222--�1236, 2011.

\bibitem{FP}
N.~Fountoulakis and K.~Panagiotou.
\newblock Sharp load thresholds for cuckoo hashing.
\newblock {\em Random Structures Algorithms}, 41(3):306--333, 2012.

\bibitem{FM}
A.~Frieze and P.~Melsted.
\newblock Maximum matchings in random bipartite graphs and the space
  utilization of cuckoo hash tables.
\newblock {\em Random Structures Algorithms}, 41(3):334--364, 2012.

\bibitem{FriezeLuczak90}
A.~M. Frieze and T.~{\L}uczak.
\newblock Edge disjoint spanning trees in random graphs.
\newblock {\em Period. Math. Hungar.}, 21(1):35--37, 1990.

\bibitem{GW4}
P.~Gao and N.~C. Wormald.
\newblock Load balancing and orientability thresholds for random hypergraphs.
\newblock In {\em {P}roceedings of the 2010 {ACM} {I}nternational {S}ymposium
  on {T}heory of {C}omputing (STOC)}, pages 97--103. 2010.

\bibitem{Goel06}
G.~Goel and J.~Gustedt.
\newblock Bounded arboricity to determine the local structure of sparse graphs.
\newblock In {\em Proceedings of the 32nd international conference on
  Graph-Theoretic Concepts in Computer Science}, WG'06, pages 159--167, Berlin,
  Heidelberg, 2006. Springer-Verlag.

\bibitem{Gusfield83}
D.~Gusfield.
\newblock Connectivity and edge-disjoint spanning trees.
\newblock {\em Information Processing Letters}, 16(2):87--89, 1983.

\bibitem{Hakimi}
S.~L. Hakimi.
\newblock On the degrees of the vertices of a directed graph.
\newblock {\em J. Franklin Inst.}, 279:290--308, 1965.

\bibitem{Itai88}
A.~Itai and M.~Rodeh.
\newblock The multi-tree approach to reliability in distributed networks.
\newblock {\em Information and Computation}, 79(1):43--59, 1988.

\bibitem{Janson}
S.~Janson and M.~J. Luczak.
\newblock A simple solution to the {$k$}-core problem.
\newblock {\em Random Structures Algorithms}, 30(1-2):50--62, 2007.

\bibitem{JLR}
S.~Janson, T.~{\L}uczak, and A.~Rucinski.
\newblock {\em Random graphs}.
\newblock Wiley-Interscience Series in Discrete Mathematics and Optimization.
  Wiley-Interscience, New York, 2000.

\bibitem{Kim}
J.~H. Kim.
\newblock Poisson cloning model for random graphs.
\newblock In {\em International {C}ongress of {M}athematicians. {V}ol. {III}},
  pages 873--897. Eur. Math. Soc., Z\"urich, 2006.

\bibitem{Lelarge}
M.~Lelarge.
\newblock A new approach to the orientation of random hypergraphs.
\newblock In {\em Proceedings of the Twenty-Third Annual ACM-SIAM Symposium on
  Discrete Algorithms (SODA)}, pages 251--264, 2012.

\bibitem{McDiarmidReed90}
C.~McDiarmid and B.~Reed.
\newblock Linear arboricity of random regular graphs.
\newblock {\em Random Structures Algorithms}, 1(4):443--445, 1990.

\bibitem{MRS}
M.~Mitzenmacher, A.~Richa, and R.~Sitaraman.
\newblock The power of two random choices: a survey of techniques and results.
\newblock In {\em Handbook of randomized computing}, volume I, II, pages
  255--312. Comb. Optim., 9, Kluwer Acad. Publ., Dordrecht, 2001.

\bibitem{MitzenmacherUpfal}
M.~Mitzenmacher and E.~Upfal.
\newblock {\em Probability and computing: randomized algorithms and
  probabilistic analysis}.
\newblock Cambridge University Press, 2005.

\bibitem{Molloy}
M.~Molloy.
\newblock Cores in random hypergraphs and {B}oolean formulas.
\newblock {\em Random Structures Algorithms}, 27(1):124--135, 2005.

\bibitem{NashWilliams61}
C.~St. J.~A. Nash-Williams.
\newblock Edge-disjoint spanning trees of finite graphs.
\newblock {\em J. London Math. Soc.}, 36:445--450, 1961.

\bibitem{Nash-Williams}
C.~St. J.~A. Nash-Williams.
\newblock Decomposition of finite graphs into forests.
\newblock {\em J. London Math. Soc.}, 39:12, 1964.

\bibitem{Palmer01}
E.~M. Palmer.
\newblock On the spanning tree packing number of a graph: a survey.
\newblock {\em Discrete Math.}, 230(1-3):13--21, 2001.
\newblock Paul Catlin memorial collection (Kalamazoo, MI, 1996).

\bibitem{PalmerSpencer95}
E.~M. Palmer and J.~J. Spencer.
\newblock Hitting time for {$k$} edge-disjoint spanning trees in a random
  graph.
\newblock {\em Period. Math. Hungar.}, 31(3):235--240, 1995.

\bibitem{PSW}
B.~Pittel, J.~Spencer, and N.~Wormald.
\newblock Sudden emergence of a giant {$k$}-core in a random graph.
\newblock {\em J. Combin. Theory Ser. B}, 67(1):111--151, 1996.

\bibitem{Schrijver03}
A.~Schrijver.
\newblock {\em Combinatorial optimization. {P}olyhedra and efficiency. {V}ol.
  {B}}, volume~24 of {\em Algorithms and Combinatorics}.
\newblock Springer-Verlag, Berlin, 2003.
\newblock Matroids, trees, stable sets, Chapters 39--69.

\bibitem{Tutte61}
W.~T. Tutte.
\newblock On the problem of decomposing a graph into {$n$} connected factors.
\newblock {\em J. London Math. Soc.}, 36:221--230, 1961.

\end{thebibliography}

\end{document}